\documentclass[final]{siamart1116}

\usepackage{amsmath}
\usepackage[sort&compress,square,numbers]{natbib}
\setcitestyle{square,citesep{;}}
\usepackage{graphicx}
\usepackage{amsfonts}
\usepackage{mathrsfs}
\usepackage{algorithmic}
\usepackage{algorithm}
\usepackage{caption}
\usepackage{tikz}
\usepackage{units}
\usetikzlibrary{matrix}
\usetikzlibrary{shapes.geometric}
\usetikzlibrary{shapes.misc}
\usetikzlibrary{decorations.pathmorphing}
\usetikzlibrary{arrows}
\usetikzlibrary{backgrounds}
\usetikzlibrary{positioning}
\usetikzlibrary{fadings}

\newtheorem{thm}{Theorem}[section]

\def\bm#1{\mbox{\boldmath$#1$}}
\newcommand{\eqdef}{\stackrel{\rm def}{=}}

\newcommand{\bb}{{\bf b}}

\newcommand{\Lb}{{\bf L}}

\newcommand{\zerob}{{\bf 0}}

\newcommand{\epsb}{\bm{\epsilon}}
\newcommand{\etab}{\bm{\eta}}
\newcommand{\R}{\mathbb{R}}
\newcommand{\N}{\mathcal{N}}

\DeclareMathOperator*{\argmin}{arg\,min}
\renewcommand{\bar}{\overline}
\renewcommand{\hat}{\widehat}

\newcommand{\eps}{\varepsilon}

\newcommand{\Gb}{{\bf G}}

\newcommand{\vect}[1]{\bm{#1}}

\begin{document}

\title{
  {Point Spread Function Estimation in X-ray Imaging with Partially Collapsed Gibbs Sampling}
  \thanks{Submitted to the editors 25 September 2017.}
}

\author{
  Kevin T.~Joyce\thanks{Signal Processing and Applied Mathematics, Nevada National Security Site, P.O. Box 98521, M/S NLV078, Las Vegas, NV, 89193-8521, USA (\email{joycekt@nv.doe.gov}, \email{LuttmaAB@nv.doe.gov}).}, 
  \and Johnathan M.~Bardsley\thanks{Department of Mathematical Sciences, University of Montana, Missoula, MT, 59812, USA (\email{bardsleyj@mso.umt.edu}).} 
  \and Aaron Luttman\footnotemark[2]
}

\maketitle

\begin{abstract}
  The point spread function (PSF) of a translation invariant imaging system is its impulse response, which cannot always be measured directly.
  This is the case in high energy X-ray radiography, and it must be estimated from images of calibration objects indirectly related to the impulse response.
  When the PSF is assumed to have radial symmetry, it can be estimated from an image of an opaque straight edge.
  We use a non-parametric Bayesian approach, where the prior probability density for the PSF is modeled as a Gaussian Markov random field and radial symmetry is incorporated in a novel way.
  Markov Chain Monte Carlo posterior estimation is carried out by adapting a recently developed improvement to the Gibbs sampling algorithm, referred to as partially collapsed Gibbs sampling.
  Moreover, the algorithm we present is proven to satisfy invariance with respect to the target density.
  Finally, we demonstrate the efficacy of these methods on radiographic data obtained from a high-energy X-ray diagnostic system at the U.S.~Department of Energy's Nevada National Security Site.
\end{abstract}



\section{Introduction}\label{sec:intro}
  Image enhancement and reconstruction is often framed within the model
  \begin{equation}\label{eq:linearInverseProblem}
  b = Ax + \varepsilon
  \end{equation}
  where $A$ is a model operator that maps a quantity of interest, $x$, to measured data, $b$, which is subject to additive measurement noise given by $\varepsilon$.
  A ubiquitous example is image deconvolution; where $x$ is an ideal un-blurred image; $b$ is the blurred data which has been corrupted by additive measurement error $\varepsilon$; and $A$ is a convolution operator whose point response is referred to as the point spread function (PSF). 
  In situations where the PSF is unknown, the same model may be used to solve the dual problem: estimate the PSF with a known image derived from some kind of calibration.
  The estimation of the PSF has its own intrinsic importance beyond its use in deconvolution, since an accurate estimate of the PSF with meaningful quantification of uncertainty serves as a useful diagnostic of the imaging system. 
  For instance, a drastic increase in the width of the PSF might indicate a malfunction in the system.

  To be more specific, the inverse problem is to estimate the PSF of $A$, say $p$, from a known calibration image $x$ which we think of as operating on $p$. 
  Expressing the convolution in \eqref{eq:linearInverseProblem} in terms of $p$, 
  \begin{equation}\label{eq:convolutionProblem}
    b = x * p + \varepsilon,
  \end{equation} 
  where `$*$' denotes the convolution operation and $Ax \eqdef x*p$.
  A direct estimate of $p$ would be available if $x$ were to represent an impulse response or point source, but in many cases, this is not feasible.
  This is acutely the case in high-energy X-ray radiography, where due to physical limitations, an impulse response cannot be obtained from calibration imagery.
  Instead, we use a vertical aperture to produce an opaque profile of an edge to estimate $p$ from the resulting integral equation.
  Several established methods use exactly this type of PSF estimation, but with parametric forms of $p$ derived from modeling the physics of the system \citep{hasen2006}. 
  These methods are often non-linear (exacerbating difficulties in the estimation and quantification of uncertainty), and parametric forms that can be solved are often not adequate to capture an accurate representation of blur that is the result of many components that act in aggregate.

  This work takes a Bayesian approach to estimation by modeling $p$ as a stochastic quantity where our a priori uncertainty is modeled with a Gauss-Markov random field.
  We incorporate measured data using a posteriori analysis, where we've modeled the measurement error with an additive likelihood model as in \eqref{eq:convolutionProblem}.
  Additionally, the prior modeling for $p$ and the parameters defining it and the measurement error are done in a hierarchical fashion, as in \cite{bardsley2012mcmc}, so that a Gibbs sampling scheme is readily applicable.
  It has been shown that this hierarchy in certain circumstances can result in highly correlated Markov chains when discretization levels limit toward the continuum \cite{agapiou2014analysis}.
  We present methods that alleviate this correlation and show that this provides effectively uncorrelated samples at an equivalent computational effort.
  Moreover, our model for the PSF provides a new method for encapsulating radial symmetry in a Gauss-Markov random field, by developing a one-dimensional precision operator that acts on the radial profile of the PSF.
  We also provide an analysis of the algorithm's convergence and computational efficiency on real and synthetic data that indicates significant efficiencies over standard Gibbs sampling as well as improvements to other newly developed methods.

  In \Cref{sec:modelingImageBlur}, we introduce a novel mathematical model for an isotropic point spread function reconstruction. 
  This results in a linear integral equation, for which the PSF can be estimated non-parametrically by discretizing the integral operator.
  In this section, we describe the hierarchical model for estimating $p$, and using Bayes' Theorem, give an explicit formulation for the posterior density of the quantities of interest.
  \Cref{sec:mcmcAlgorithms} outlines three MCMC approaches to analyzing the posterior. 
  We first outline the standard Gibbs sampling approach, then present a recently studied approach called \emph{marginal then conditional} (MTC) sampling \cite{fox2015fast}, and show it's relationship to Gibbs sampling.
  Then, we present the partial collapsed Gibbs sampler as it applies to hierarchical models, and show how it is related to the Gibbs and MTC samplers.
  Finally, \Cref{sec:results} compares each algorithm numerically on synthetic data and on actual measured PSF data a from high-energy X-ray imaging system at the U.S. Department of Energy's Nevada National Security Site..

\section{Modeling Image Blur} \label{sec:modelingImageBlur} 
  When image blur is translation invariant, it can be modeled as convolution with a PSF that represents the impulse response of the system \citep{grafakos2014}.
  A direct estimate is available by taking a calibration image representing the impulse response \citep{jain1989,roggemann1996imaging}; for example, in astronomical imaging, it is often estimated by imaging single stars which approximate point sources \citep{fox2015fast,jain1989,roggemann1996imaging}.
  In our applications, imaging a point source is not feasible, so instead, we model the system response from an image of an aperture that retains the extent information of the action of blurring.
  The measurement is inherently indirect, and requires the solution of an inverse problem.
  If the action of the blur is isotropic, then the PSF will be radially symmetric, and it can be estimated as a one-dimensional function of distance.
  More specifically, let $\vect s = (s_1,s_2)\in \R^2$ denote a position in space indexing the intensity of the response of the blurring operator (denote similarly $\vect s'$ for the domain of the PSF). 
  If $k(\vect s')$ is the value of the PSF, then it is given by a function in one variable through the composistion 
  \begin{equation}
    k(\vect s') = p( \|\vect s'\|_{\R^2} )
  \label{eq:radialSymmetryRepresentation}
  \end{equation}
  where $\|\vect s'\|_{\R^2}$ denotes the Euclidean distance in $\R^2$.
  In this case, we use a beveled vertical aperture which produces a uniformly opaque vertical edge at a known fixed location in the imaging plane.
  See \Cref{fig:edgePicture}.

  \begin{figure}[ht] 
    \begin{center}
      \begin{tikzpicture}[scale=.8,every node/.style={minimum size=1cm},on grid]
        \def\myxslant{0.1}
        \def\myyslant{-0.4} 
        \begin{scope}[
          xshift=40,
          every node/.append style={
          xslant=\myxslant,yslant=\myyslant},xslant=\myxslant,yslant=\myyslant
          ]
          \draw (0,0) rectangle (2.8,2.2);
        \end{scope}

        \begin{scope}[
          yshift=-20,
          every node/.append style={xslant=\myxslant,yslant=\myyslant},
          xslant=\myxslant,yslant=\myyslant
          ]
          \draw[x=.314cm,y=.2cm,z=.2cm,thick,-latex,red] (0,0,0)
            sin ++(0,1,1) cos ++(0,-1,1) sin ++(0,-1,1) cos ++(0,1,1)
            sin ++(0,1,1) cos ++(0,-1,1) sin ++(0,-1,1) cos ++(0,1,1)
            sin ++(0,1,1) cos ++(0,-1,1) sin ++(0,-1,1) cos ++(0,1,1);
          \draw[x=.314cm,y=.2cm,z=.2cm,thick,-latex,red] (-2,0,0)
            sin ++(0,1,1) cos ++(0,-1,1) sin ++(0,-1,1) cos ++(0,1,1)
            sin ++(0,1,1) cos ++(0,-1,1) sin ++(0,-1,1) cos ++(0,1,1)
            sin ++(0,1,1) cos ++(0,-1,1) sin ++(0,-1,1) cos ++(0,1,1);
          \draw[x=.314cm,y=.2cm,z=.2cm,thick,-latex,red] (-4,0,0)
            sin ++(0,1,1) cos ++(0,-1,1) sin ++(0,-1,1) cos ++(0,1,1)
            sin ++(0,1,1) cos ++(0,-1,1) sin ++(0,-1,1) cos ++(0,1,1)
            sin ++(0,1,1) cos ++(0,-1,1) sin ++(0,-1,1) cos ++(0,1,1);

          \pgfmathsetmacro{\cubex}{2}
          \pgfmathsetmacro{\cubey}{2.5}
          \pgfmathsetmacro{\cubez}{1}
          \draw[black,fill=gray,opacity=.75] (3.7,3,0) -- ++(-\cubex,0,0) -- ++(0,-\cubey,0) -- ++(\cubex,0,0) -- cycle;
          \draw[black,fill=gray,opacity=.75] (3.7,3,0) -- ++(0,0,-\cubez) -- ++(0,-\cubey,0) -- ++(0,0,\cubez) -- cycle;
          \draw[black,fill=gray,opacity=.75] (3.7,3,0) -- ++(-\cubex,0,0) -- ++(0,0,-\cubez) -- ++(\cubex,0,0) -- cycle;

          \draw[x=.314cm,y=.2cm,z=.2cm,thick,-latex,red] (4,0,0)
            sin ++(0,1,1) cos ++(0,-1,1) sin ++(0,-1,1) cos ++(0,1,1)
            sin ++(0,1,1) cos ++(0,-1,1) sin ++(0,-1,1) cos ++(0,1,1);
          \draw[x=.314cm,y=.2cm,z=.2cm,thick,-latex,red] (2,0,0)
            sin ++(0,1,1) cos ++(0,-1,1) sin ++(0,-1,1) cos ++(0,1,1)
            sin ++(0,1,1) cos ++(0,-1,1) sin ++(0,-1,1) cos ++(0,1,1);
          \end{scope}

          \begin{scope}[
              xshift=180,
              every node/.append style={xslant=\myxslant,yslant=\myyslant},
              xslant=\myxslant,yslant=\myyslant
            ]
            \draw (0,0) rectangle (2.8,2.2);
            \tikzfading[name=fade left,left color = transparent!100,right color = transparent!0]
            \draw[path fading=fade left,fading transform={rotate=-30},fill=black] (.9,0) rectangle (1,2.2);
            \draw[fill=black] (1,0) rectangle (2.8,2.2);
        \end{scope}

        \begin{scope}[
          xshift=300,
          every node/.append style={xslant=\myxslant,yslant=\myyslant},
          xslant=\myxslant,yslant=\myyslant
          ]
          \draw (0,0) rectangle (2.8,2.2);
          \tikzfading[name=fade left,left color = transparent!100,right color = transparent!0]
          \draw[path fading=fade left,fading transform={rotate=-30},fill=black] (.9,0) rectangle (1,2.2);
          \draw[fill=black] (1,0) rectangle (2.8,2.2);
          \draw[step=1mm, black] (0,0) grid (2.8,2.2); 
        \end{scope}
        \node at (2,2.5) (label1) {Opaque Edge};
        \draw[-latex,thick] (2,-2) to node[below] {Image System Response} (7.5,-2);
        \node at (4.65,-3.4) (math1) {$\displaystyle{b(\vect s)=\iint_E {p(\|\vect s-\vect s'\|_{\R^2})}\,d\vect s'}$};

        \node at (7,2.5) (label3) {Blurred Profile};

        \node at (12,2.5) (label2) {Recorded Data};
        \draw[-latex,thick] (9,-2) to node[below]{Measurement error} (13.2,-2);
        \node at (11,-3.4) (math1) {$+\vect \eps \sim N(\vect 0,\lambda^{-1} \vect I)$};
      \end{tikzpicture}
      \caption{ A schematic of the measurement model for an X-Ray image of an edge. An opaque block whose profile is indicated by $E$ blocks light on the half plane to produce a blurred edge.}\label{fig:edgePicture}
      \end{center}
  \end{figure}
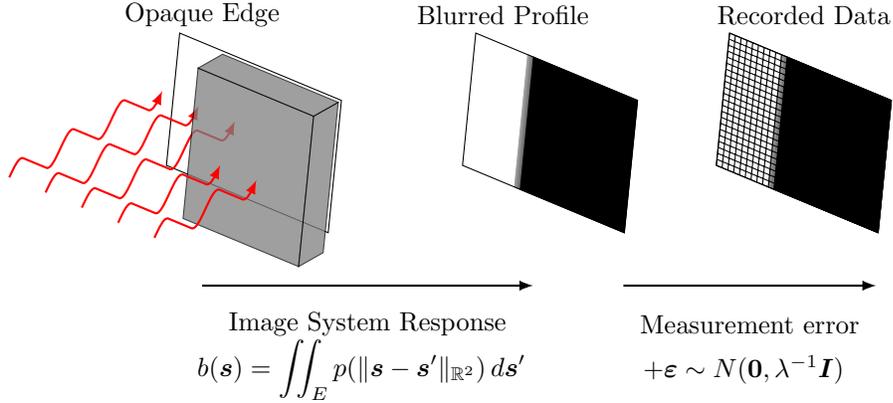

  To model this mathematically, let $E$ denote the half plane, and note the characteristic function, $\chi_E(s_1',s_2')$, depends only on the horizontal coordinate of $s_1'$, so the convolution in \eqref{eq:convolutionProblem} can be written as
  \begin{equation}
    b(\vect s) = \iint p( \|\vect s'\|_{\R^2})\chi_{(0,\infty)}(s_1-s_1')d\vect s' + \varepsilon,
    \label{eq:psfForwardModel}
  \end{equation}
  where $b$ represents the noisy and blurred measured edge, and $p$ is the radial profile of the PSF.
  Since \eqref{eq:psfForwardModel} does not depend on $s_2$ (the edge has vertical translation symmetry), the output of the integral operator in \eqref{eq:psfForwardModel} is only a function of $s_1$.
  Thus, we may represent $b(s_1,s_2) = b(s_1)$, and for the sake of clarity, we denote $b(s_1) = b(s)$.
  Further, denote $r=\|\vect s'\|_{\R^2}$.
  In this way, the inverse problem has been reduced to estimating functions on subsets of $\R$ -- that is, estimate the radial profile of the PSF, $p(r)$, from a horizontal cross-section of an image of a blurred edge, $b(s)$.
  Using the change of variables $s_1' = r\cos v, s_2'= r\sin v$ in \eqref{eq:psfForwardModel} and integrating out the $v$ variable results in the integral operator on the radial profile
  \begin{equation}
    b(s) = \int_0^\infty p(r) g(s,r) r dr, \label{eq:fredholmEquation}
  \end{equation}
  where
  \begin{equation}
    g(s,r) = \left\{
    \begin{array}{lc}
      0                        &\quad  s  < - r\\
      2(\pi - \cos^{-1}(s/r))  & |s| \le r\\
      2\pi                     &\quad  s  > r.
    \end{array}\right.  \label{g_form}
  \end{equation}
  This situation is illustrated in \Cref{fig:edgePicture,fig:edgeData}.

  Observe that $g(s,r)$ is continuous, but not differentiable (it has a discontinuity in its directional derivatives across the line $r=s$).
  Again, since the operator is compact, its discretization results in a matrix with singular values that cluster near zero \citep{hansen2010}, as evidenced by \Cref{fig:svd}.
  Hence, discrete estimation in the presence of measurement error will be unstable \citep{hansen2010}.
  %
  %
  For such ill-posed problems, prior knowledge about the solution must be incorporated to make the problem well-posed.
  By representing the PSF as a radial profile $p(r)$, the space and geometry for the domain of the model operator must reflect this representation, and prior notions of smoothness of the PSF must be expressed appropriately in this space.
  That is, since $p$ depends on the distance $r=\|\vect s'\|_{\R^2}$, integration-based regularization operators (the viewpoint in \citep{vogel2002}), and precision operators in Gaussian based probabilistic frameworks (the viewpoint in \citep{stuart2010}) will involve a change of variables.
  Both of these methods typically result in solving the penalized least square problem
  \begin{equation} 
    p_{\lambda,\delta}=\argmin_{p}\left\{\lambda \|\mathcal Gp - b\|_{L^2}^2+\delta F(p)\right\}, \label{eq:leastSquareSolution}
  \end{equation}
  where $F$ is the corresponding regularizing norm.
  A discrete version of \eqref{eq:leastSquareSolution} is derived in a probabalistic framework in the next section.

  For our application, imposing Laplacian based smoothness on the PSF is an appropriate prior assumption.
  Hence, if one denotes the 2D Laplacian by $\Delta$, then $r = \sqrt{{s_1}^2 + {s_2}^2}$ implies
  \begin{align}
     \Delta(p \circ r)
     = r^{-1}\cdot \left( \frac d{dr}\left(r \cdot\frac {dp}{dr}\right) \right).
  \end{align}
  Note that this is the radial component of the Laplacian in two-dimensional polar coordinates.  
  Denote the differential operator 
  \begin{equation}
    Rp \eqdef \frac d{dr}\left(r \cdot\frac {dp}{dr}\right). \label{eq:radialDifferentialOperator}
  \end{equation}
  Defining $F$ in terms of the $L^2$ inner product, induces a similar change of variables; i.e.
  \begin{align}
      F_\alpha(p) \eqdef \alpha \left\langle p\circ r, \Delta^n (p\circ r)\right\rangle_{L^2} 
      &= 2\pi\alpha \int_0^\infty p(r)\cdot\left(R^np(r)\right)\cdot r^{1-n}dr. \label{eq:radialProfilePrior}
  \end{align}
  So, Laplacian regularization of order $n$ smoothness on the PSF induces a regularization operator on its radial representation of the form $r^{1-n} R^n p$.
  A more rigorous development of these notions is carried out in \cite{joyce2016psf}.

  For boundary conditions, we assume regularity of the PSF at the origin so that
  \begin{equation}
    \left.\frac{d}{dr}\right|_{0^+} p(r) = 0.\label{eq:leftContinuousBoundary}
  \end{equation}
  We also assume that the PSF decays away from the origin such that for any $k$
  \begin{equation}
    \lim_{r\to\infty} r^kp(r) = 0,
    \label{eq:rightContinuousBoundary}
  \end{equation} 
  which, when discretized, we assume the imaging field of view is such that the radial profile is sufficiently small in magnitude to assume a zero right boundary condition on the domain of the solution.

  In the probabilistic framework, the solution to \eqref{eq:leastSquareSolution} is equivalent to a maximum a posteriori (MAP) estimate when the PSF is assumed to be a Gaussian, and taking $n=2$ guarantees that the corresponding prior covariance operator is trace class \citep{stuart2010}.
  Since data and estimates are inherently discrete quantities, we proceed by discretizing \eqref{eq:fredholmEquation} and \eqref{eq:radialProfilePrior}.

  \begin{figure}[htbp!]
  \begin{center}
    \includegraphics[width=.45\textwidth]{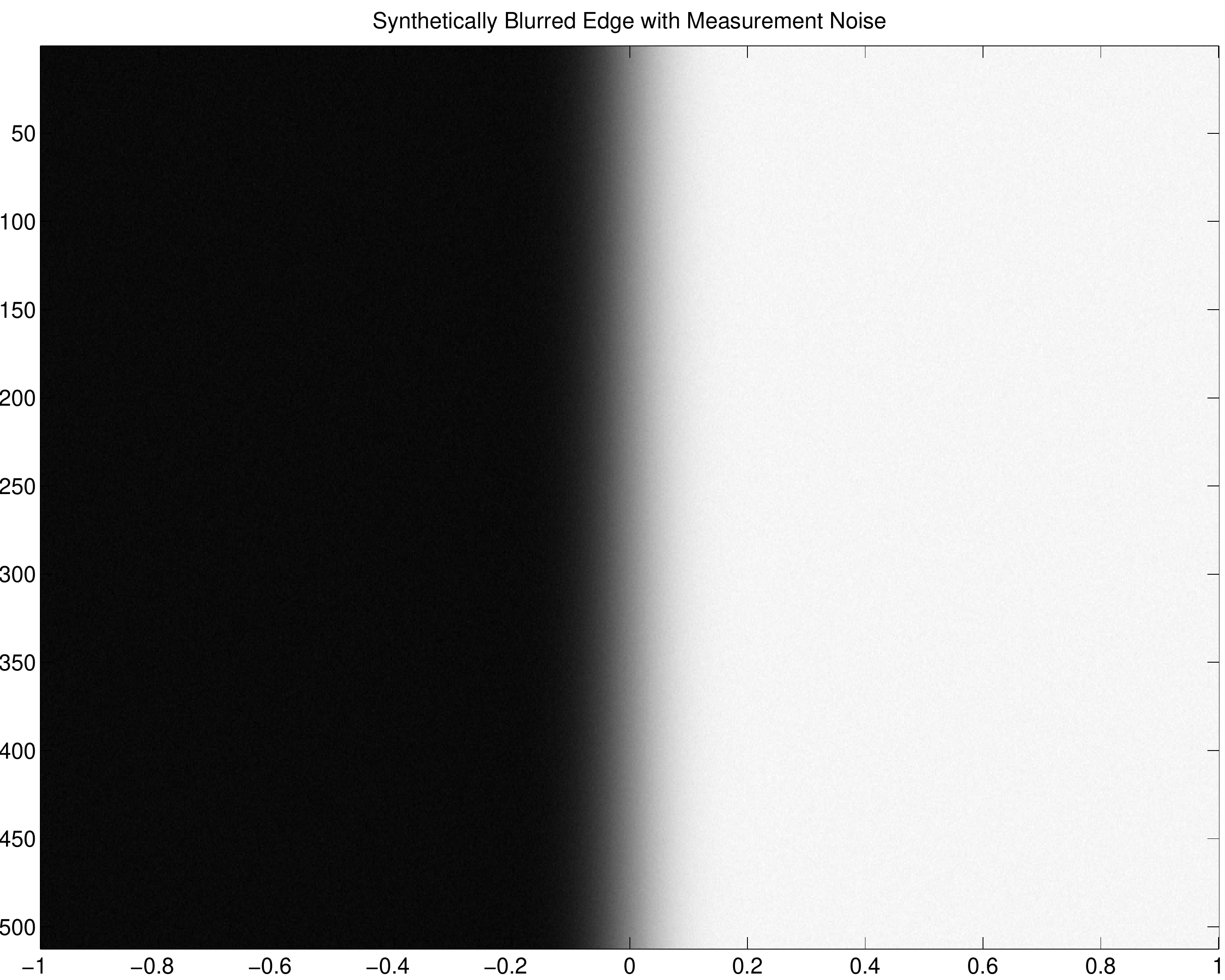}
    \includegraphics[width=.45\textwidth]{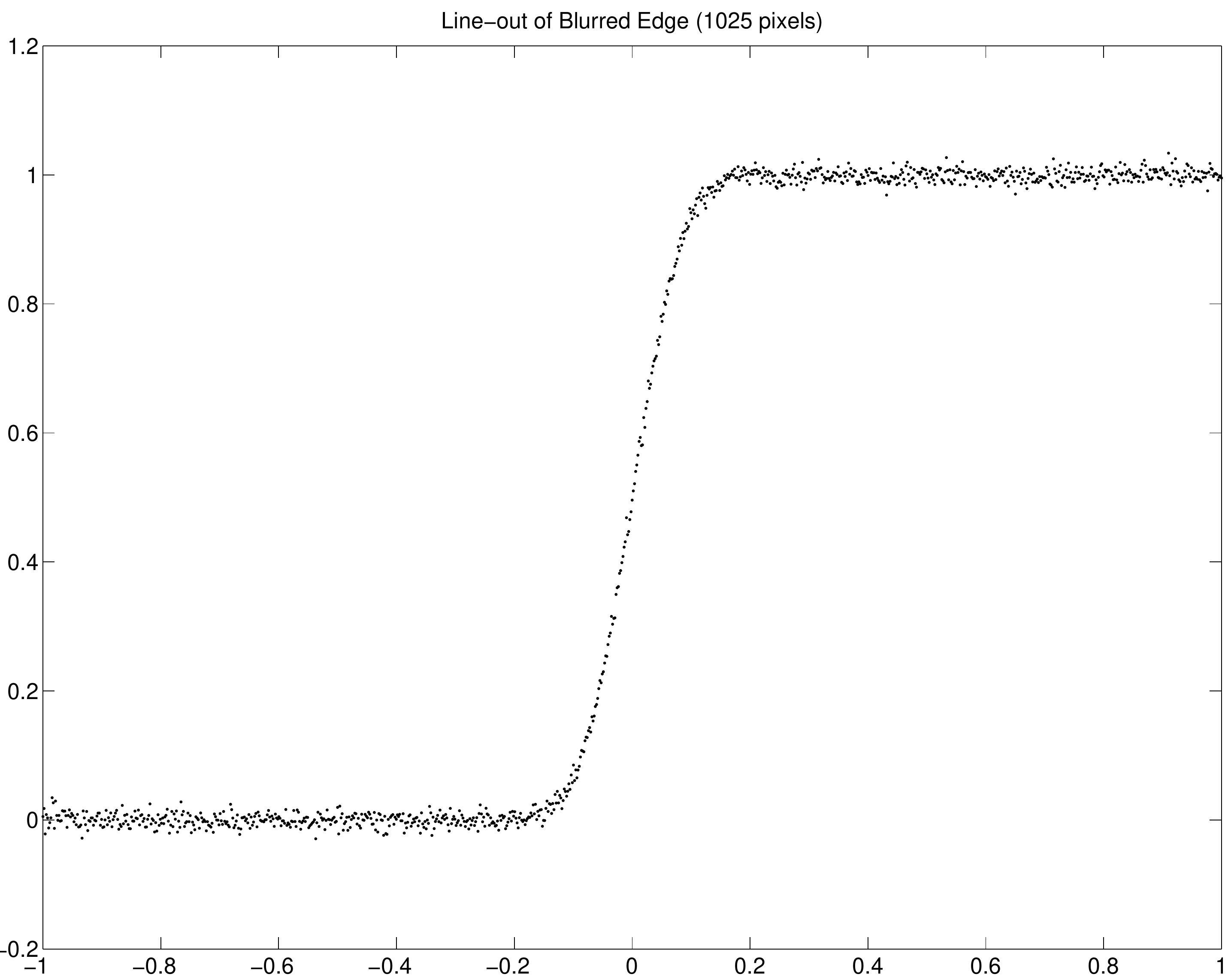}
    \caption{A synthetically blurred edge with simulated measurement error and a line-out (horizontal cross-section) from the data.} \label{fig:edgeData}
  \end{center}
  \end{figure}

  \begin{figure}[htbp!]
  \begin{center}
    \includegraphics[width=.45\textwidth]{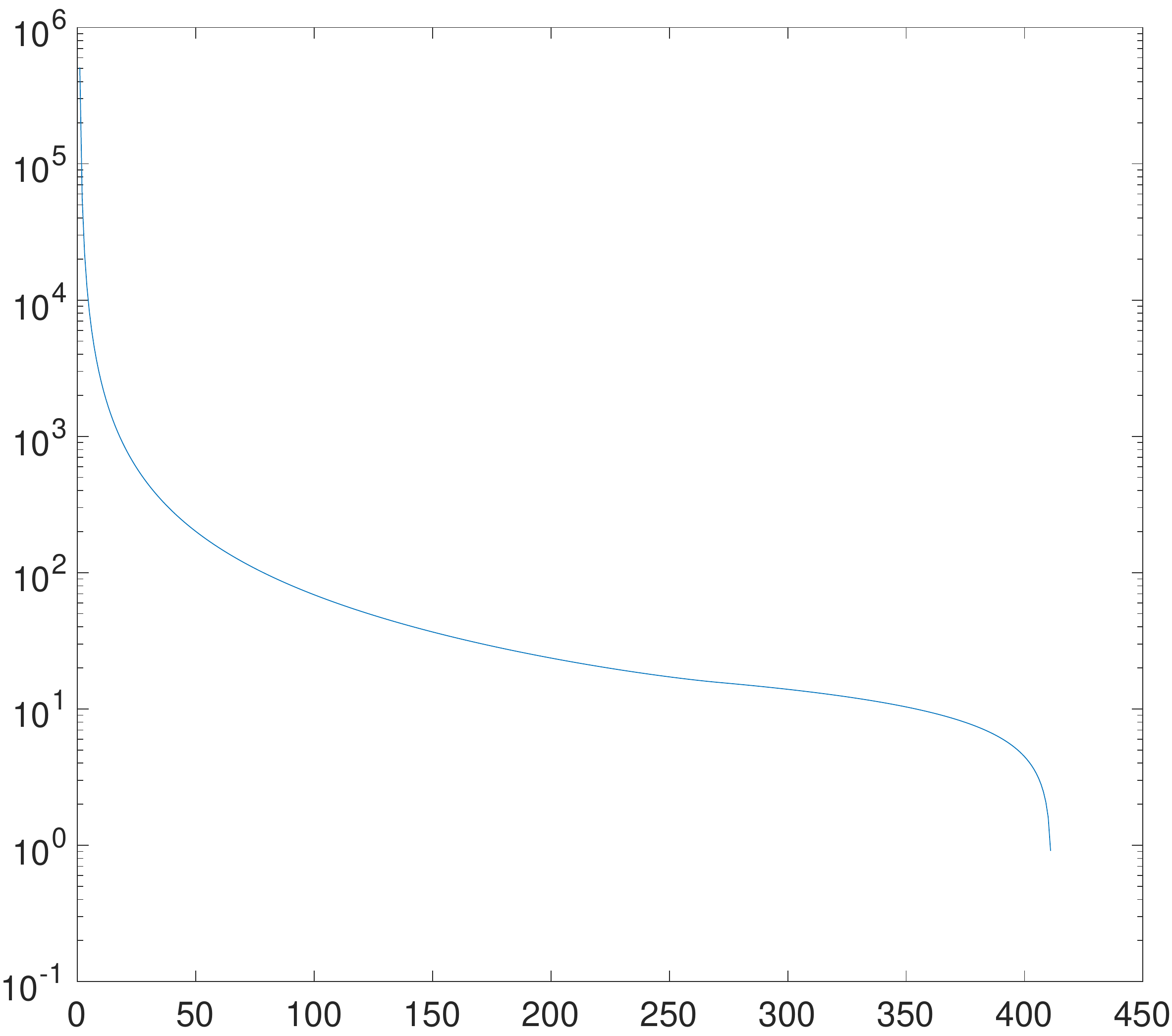}
    \includegraphics[width=.45\textwidth]{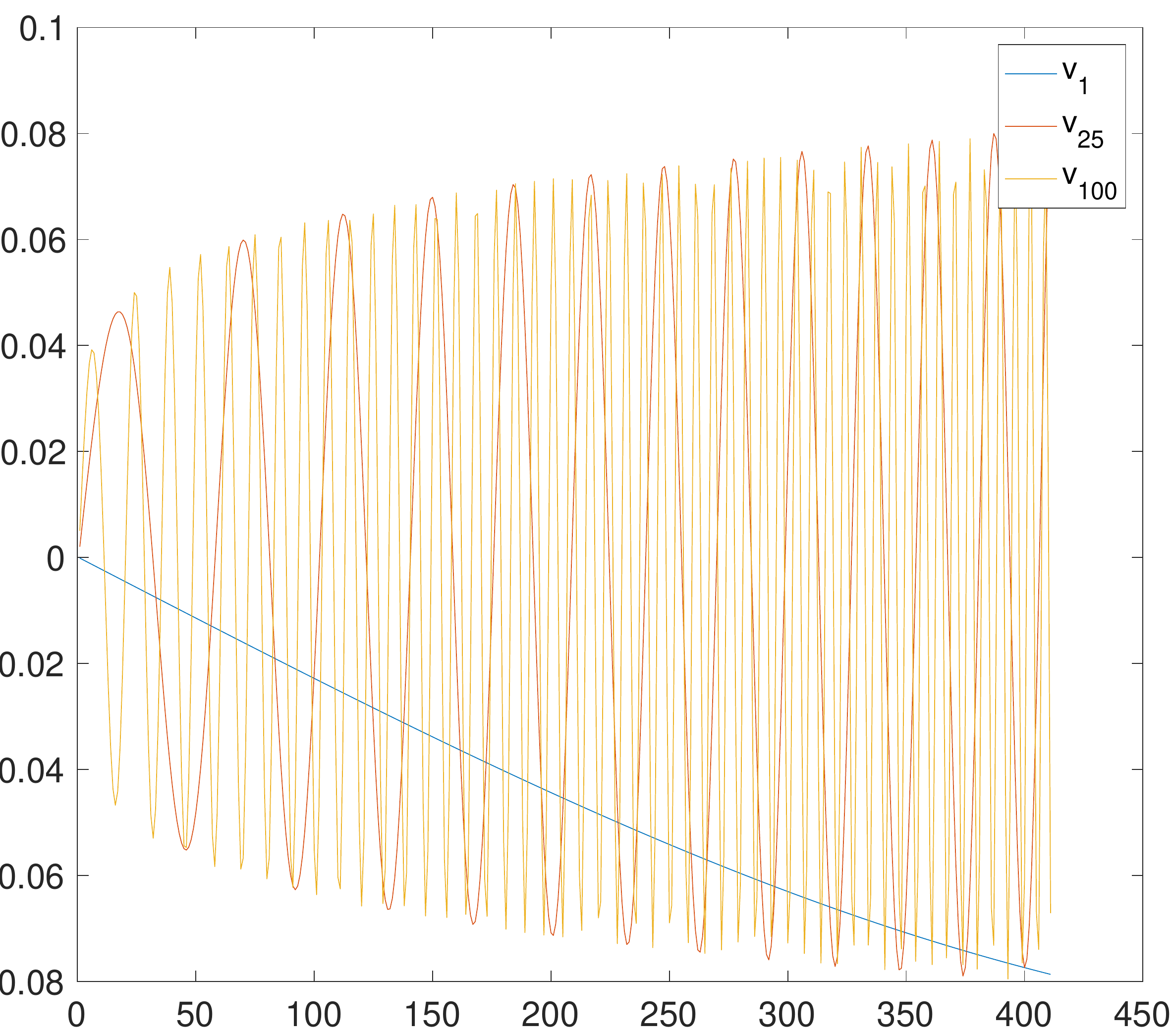}
    \caption{Plots of the singular values and a selection of the right singular vectors of the discretization of the forward edge blur operator $\mathcal G$.} \label{fig:svd}
  \end{center}
  \end{figure}

  \subsection{Numerical discretization} 
    The data are intensity values of image pixels from a fixed horizontal cross-section of the edge, sampled at $M=2N+1$ points $s_i\in [-1,1]$ with $s_i = i/N$ and $-N\le i \le N$ such that $s_0$.
    Denote the grid spacing as $h \eqdef 1/N$ and the vector of data as $\vect b \in \R^{2N+1}$ with entries $b_i \eqdef b(s_i)$. 

    Since $g$, the integral kernel in \eqref{eq:fredholmEquation}, is supported on $\{(r,s): r\ge -s,r\ge 0\}$, the bounds of integration depend on $s_i$. 
    Hence, a midpoint quadrature rule for $b(s_i)$ places $r_j$ on the midpoints, i.e., $r_j = h(j - 1/2)$ for $1\le j \le N$, $\Gb_{ij} \eqdef g(s_i,r_j)$, and $\vect p_j = p(r_j)$ gives
    \begin{equation}
      \int_{s_i}^\infty p(r)g(s_i,r)dr\approx h\sum_{j=1}^N \Gb_{ij} \vect p_j \label{eq:discretization}
    \end{equation}
    Note that due to the symmetry of the integration kernel $g$, imposing the same sampling resolution on $\vect p$ as $\vect b$ results in $\Gb$ being a $(2N+1) \times N$ matrix.  

    The differential operator $R$ in \eqref{eq:radialDifferentialOperator} is discretized using centered differencing \cite{morton2005numerical}. 
    Explicitly, for $r_{j\pm1/2} = r_j \pm h/2$, the matrix stencil for $\vect R$ is
    \begin{equation}
      \frac{1}{h^2}
      \left[\begin{array}{ccc}
        -(r_{j-3/2} + r_{j-1/2}) & r_{j-1/2} & 0             \\
        r_{j-1/2} & -(r_{j-1/2} + r_{j+1/2}) & r_{j+1/2}     \\
        0 & r_{j+1/2} & -(r_{j+1/2} + r_{j+3/2}) \\
      \end{array}\right]
      \left[\begin{array}{c}
        p_{j-1} \\
        p_{j}   \\
        p_{j+1} \\
      \end{array}\right].
      \label{laplacian_discretization_stencil}
    \end{equation}

    The left boundary condition given in \eqref{eq:leftContinuousBoundary} and radial symmetry implies that the discretization of $\bm R$ has a reflective left boundary condition, hence
    \begin{equation}
      [\bm R\bm p]_1 = 2 r_{1/2} p_1.
    \end{equation}
    We assume that the imaging field of view is sufficiently large so that the right boundary condition in \eqref{eq:rightContinuousBoundary} is satisfied to numerical precision; i.e., 
    \begin{equation}
      [\bm R\bm p]_M = r_{N+1/2} p_N. 
    \end{equation}
    Finally, the discrete precision matrix for the prior is given by $\bm L \eqdef \bm r^{-1} \odot \bm R^2$ (since $n=2$), the coordinate-wise multiplication of reciprocals of the grid points $r_j$ composed with $\bm R^2$.

  \subsection{Bayesian inference with hierarchical modeling} 
    Our approach is to form a probabilistic model to estimate the unknown discrete representation of the PSF as well as the parameters involved in defining each of the distributions.
    That is, in addition to modeling uncertainty in $\vect p$ with a random field, it has become common to develop hierarchical models to let the data inform the level of regularization \citep{higdon2006primer,bardsley2012mcmc,howard2016bayesian,howard2014samplilng,calvetti2007introduction,kaipo2005}. 
    This analyisis will be conducted on the the discrete model 
    \begin{align}
    \bb=\Gb\vect p+\epsb, \label{eq:discreteInverseProblem}
    \end{align}
    where $\epsb$ models the discrete measurement error.

    Assuming that the discrete measurement error is independent Gaussian noise $\vect \eps \sim \N(\vect 0,\lambda^{-1}\vect I)$, the likelihood is a probability density satisfying
    \begin{equation}
    \pi(\bb|\vect p,\lambda)\propto \lambda^{M/2}\exp\left(-\frac{\lambda}{2}\Vert\Gb\vect p-\bb\Vert^2\right).  \label{eq:likelihood}
    \end{equation}
    When the quantity of interest is assumed to have a discrete Gauss-Markov random field prior with precision $\delta \vect L$, then $\vect p\sim \N(\zerob,(\delta\Lb)^{-1})$, in which case the prior density satisfies
    \begin{equation}
    \pi(\vect p|\delta)\propto \delta^{N/2}\exp\left(-\frac{\delta}{2}\vect p^T\Lb\vect p\right), \label{eq:prior}
    \end{equation}
    where the inverse covariance $\delta\Lb$ is the previously derived radially symmetric squared-Laplacian scaled by $\delta$.
    Applying Bayes' theorem, the probability density function for $\vect p|\bb,\lambda,\delta$ is given 
    \begin{align}
    \pi(\vect p|\bb,\lambda,\delta)
      &\propto \pi(\bb|\vect p,\lambda)\pi(\vect p|\delta) \nonumber\\
      &= \lambda^{M/2}\delta^{N/2}\exp\left(-\frac{\lambda}{2}\Vert\Gb\vect p-\bb\Vert^2-\frac{\delta}{2}\vect p^T\Lb\vect p\right). \label{eq:prePosterior}
    \end{align}
    The {\em maximum a posteriori} (MAP) estimator is the maximizer of \eqref{eq:prePosterior}, which is also the minimizer $-\ln\,\pi(\vect p|\bb,\lambda,\delta)$.
    By expanding the inner products and centering the quadratic form in terms of $\vect p$, the density in \eqref{eq:prePosterior} can be shown to be the Gaussian
    \begin{equation}
      \vect p|\lambda,\delta,\bb
        \sim \N\left((\lambda\Gb^T\Gb+\delta\Lb)^{-1}\lambda\Gb^T\bb,(\lambda\Gb^T\Gb+\delta\Lb)^{-1}\right).\label{eq:prePosteriorDensity}
    \end{equation}

    Following the Bayesian paradigm, the unknown parameters $\lambda$ and $\delta$ are also modeled as random quantities.
    A common hyper-prior model for these parameters is a Gamma distribution because of the mutual conjugacy it shares with the Gaussian, making their conditional distributions with respect to the data easy to simulate \citep{gelman2014bayesian}.
    Moreover, the flexibility of the Gamma distribution allows for a relatively unobtrusive hyper-prior probability density when little a priori information about $\lambda$ and $\delta$ is available.
    Thus,
    \begin{align}
    \pi(\lambda) &\propto \lambda^{\alpha_\lambda-1}\exp(-\beta_\lambda\lambda),\label{hyper_lambda}\\
    \pi(\delta) &\propto \delta^{\alpha_\delta-1}\exp(-\beta_\delta\delta).\label{hyper_delta}
    \end{align}
    We choose the hyper-prior parameters to be $\alpha_\lambda=\alpha_\delta=1$ and $\beta_\lambda=\beta_\delta=10^{-6}$ so that the prior distributions of $\lambda$ and $\delta$ cover a broad range of values that have been estimated for similar problems \cite{higdon2006primer,bardsley2012mcmc,howard2016bayesian}.
    Hence, the full posterior probability density function for $(\vect p,\lambda,\delta|\bb)$ is
    \begin{align}
      &\pi(\vect p,\lambda,\delta|\bb) 
        \propto \pi(\bb|\vect p,\lambda)\pi(\vect p|\delta)\pi(\lambda)\pi(\delta)\nonumber\\
        \quad&=\lambda^{M/2+\alpha_\lambda-1}\delta^{N/2+\alpha_\delta-1}\exp\left(-\frac{\lambda}{2}\Vert\Gb\vect p-\bb\Vert^2-\frac{\delta}{2}\vect p^T\Lb\vect p-\beta_\lambda\lambda-\beta_\delta\delta\right).
    \label{eq:posterior}
    \end{align}

\section{MCMC algorithms for posterior inference} \label{sec:mcmcAlgorithms} 
  The primary goal of this work is to draw statistical inference on the joint variable $\vect p,\lambda,\delta|\vect b$ by characterizing the joint-posterior density in \eqref{eq:posterior}.  
  Due to the hierarchical modeling of $\lambda$ and $\delta$, the density in \eqref{eq:posterior} does not have a common distributional form, so explicit characterization is not readily available.
  Monte Carlo methods that utilize the invariance of a Markov process, so called Markov Chain Monte Carlo (MCMC), have become standard because of their computational efficiency and broad applicability.
  Gibbs sampling, in particular, has found great utility \cite{higdon2006primer,bardsley2012mcmc,fowler2016stochastic,howard2016bayesian} due to its direct application to hierarchical modeling with conjugate random variables and ease of implementation with relatively little tuning.
  Moreover, investigating the Gibbs sampler and its convergence properties have been explored in \cite{agapiou2014analysis,bardsley2016metropolis,fox2015fast,van2008partially,van2015metropolis}.
  This work builds upon this literature, by providing an algorithm that is shown to empirically improve the convergence of Gibbs sampling.

  The application of partial collapse to a general Gibbs sampling scheme was studied in \citep{van2008partially}, and they show that partial collapse must be done with care, since the resulting Markov chain may no longer be invariant.
  The loss of invariance is case dependent, and this work provides a proof of the invariance of the Markov chain and directly addresses the potential pitfalls alluded to in \cite{van2008partially}.  

  \subsection{Gibbs sampling} 
    In a Gibb's sampling framework, the \emph{full conditionals} of each component of the posterior density are used to compute samples of the posterior density.
    We have already characterized $\pi(\vect p| \lambda, \delta)$ in \eqref{eq:prePosteriorDensity} in establishing the MAP estimator for fixed $\lambda$ and $\delta$.  
    The full conditionals for $\lambda$ and $\delta$ are computed by removing proportional terms from \eqref{eq:posterior}, and observing that the remaining conditional densities satisfy
    \begin{align}
      \pi(\lambda|\vect p,\delta,\bb)
        &\propto \lambda^{M/2+\alpha_\lambda-1}\exp\left(\left[-\frac{1}{2}\Vert\Gb\vect p-\bb\Vert^2-\beta_\lambda\right]\lambda\right),\label{eq:lambdaConditional}\\
      \pi(\delta|\vect p,\lambda,\bb)
        &\propto \delta^{N/2+\alpha_\delta-1}\exp\left(\left[-\frac{1}{2}\vect p^T\Lb\vect p-\beta_\delta\right]\delta\right).\label{eq:deltaConditional}
    \end{align}
    These are scalings and shifts of the corresponding hyper-priors, and are thus $\Gamma$-distributed.

    With each of the full conditional distributions characterized in \eqref{eq:prePosteriorDensity}, \eqref{eq:lambdaConditional}, and \eqref{eq:deltaConditional}, the requisite simulations required to establish the hierarchical Gibbs sampler are known, and steps for the Gibbs sampler are given in \Cref{alg:hierarchicalGibbs}.
    This algorithm has been used successfully for other inverse problems applications in computational imaging, and its efficacy has been demonstrated in \cite{bardsley2012mcmc,fowler2016stochastic,howard2014samplilng}.

    {
    \medskip 
    \captionof{algorithm}{Hierarchical Gibbs Sampler for PSF reconstruction} \label{alg:hierarchicalGibbs}
    \hrule
    Given $\lambda_k,\delta_k$ and $\vect p^k$, simulate
    \begin{enumerate}
    \item[1.] $\lambda_{k+1}\sim \Gamma\left(M/2+\alpha_\lambda,\frac{1}{2}\Vert\Gb\vect p^{k}-\bb\Vert^2+\beta_\lambda\right)$;
    \item[2.] $\delta_{k+1}\sim \Gamma\left(N/2+\alpha_\delta,\frac{1}{2}(\vect p^{k})^T\Lb\vect p^{k}+\beta_\delta\right)$;
    \item[3.] $\vect p^{k+1}\sim \N\left((\lambda_{k+1}\Gb^T\Gb+\delta_{k+1}\Lb)^{-1}\lambda_{k+1}\Gb^T\bb,(\lambda_{k+1}\Gb^T\Gb+\delta_{k+1}\Lb)^{-1}\right)$.
    \end{enumerate}
    \hrule
    \medskip
    } 

    In steps 1 and 2, several well-established algorithms exist for simulating draws from $\Gamma$ distributions, and are readily available in most statistical software packages \cite{marsaglia2000gamma}.
   The simulation in step 3 can be achieved by solving the system 
    \begin{equation} \label{sampleX}
      \vect p^{k+1} = \Big(\lambda_k \Gb^T \Gb + \delta_k\Lb\Big)^{-1} (\lambda_k\Gb^T \bb + \etab),\quad \etab\sim\N(\zerob,\lambda_k \Gb^T \Gb + \delta_k\Lb).
    \end{equation}

    In a generic Gibbs sampling framework, any permutation of the steps is still a proper algorithm in the sense that they all produce chains that are invariant with respect to the joint random variable. 
    However, in the hierarchical framework, there is a natural ordering that separates the hierarchical variables and the quantity of interest, which we show in the next section.
    Moreover, when partial collapse is applied, a permutation actually leads to an algorithm where invariance is lost.


    \subsection{Blocking and MTC Sampling}

    Observe that the conditional densities in steps 1 and 2 of \Cref{alg:hierarchicalGibbs} are conditionally independent. 
    That is, the normalizing constant in \eqref{eq:lambdaConditional} is an integral that does not depend on $\delta$ and vice versa for $\lambda$ in \eqref{eq:deltaConditional}; hence, 
    \begin{equation}
      \pi(\lambda,\delta|\vect p,\vect b) = \pi(\lambda|\vect p,\vect b) \pi(\delta|\vect p,\vect b).
    \label{eq:conditionalIndependence}
    \end{equation}
    This has the effect that the hierarchical Gibbs sampler in \Cref{alg:hierarchicalGibbs} naturally blocks $\delta$ and $\lambda$.
    Explicitly, if we denote $\vect \theta = (\lambda,\delta)$, then \Cref{alg:hierarchicalGibbs} is equivalent to \Cref{alg:twoStageGibbs}:

    {
    \medskip 
    \captionof{algorithm}{Two-stage Gibbs Sampler for PSF reconstruction} \label{alg:twoStageGibbs}
    \hrule
    Given $\vect \theta^k = (\lambda_k,\delta_k)$ and $\vect p^k$, simulate
    \begin{enumerate}
      \item[1.] $\vect \theta^{k+1} \sim \pi(\lambda,\delta|\vect p^k,\vect b)$ \quad by \eqref{eq:conditionalIndependence}
      \item[2.] $\vect p^{k+1} \sim \pi(\vect p|\lambda_{k+1},\delta_{k+1},\vect b)$ \quad by \eqref{eq:prePosteriorDensity}
    \end{enumerate}
    \hrule
    \medskip 
    }
    Note that any non-trivial permutation with step 3 of \Cref{alg:hierarchicalGibbs} makes it impossible to block $\lambda$ and $\delta$, and the methods are no longer equivalent.

    When Gibbs sampling happens in two stages, the two separate components $\vect \theta^k$ and $\vect p^k$ are themselves Markov chains whose stationary distribution is given by the corresponding marginalized density \cite[Chapter 9]{robert2013monte}.
    Moreover, the transition kernel associated with $\{\vect \theta^k\}$ is 
    \begin{equation}
      \int_{\R^n} \pi(\vect p| \lambda,\delta,\vect b) \pi(\lambda',\delta'|\vect p,\vect b) d\vect p. \label{eq:thetaTransition}
    \end{equation}
    This makes the analysis divide naturally into considering the quantity of interest, $\vect p$, and the hierarchical parameters, $\lambda$ and $\delta$.

    \begin{figure}[htbp!]
    \begin{center}
      \includegraphics[width=\textwidth]{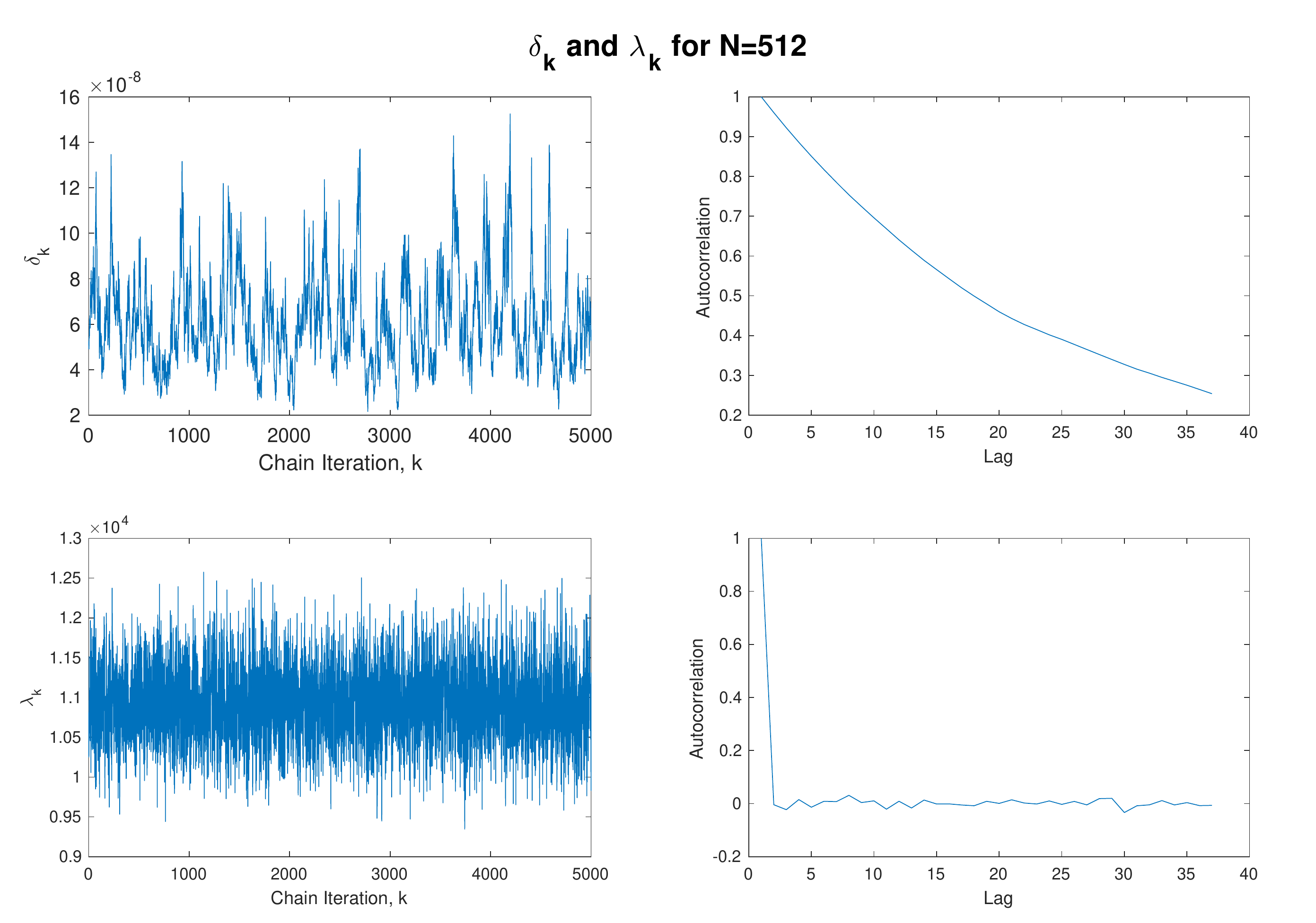}
      \caption{The $\delta_k$ and $\lambda_k$ components of the Markov chain resulting from the hierarchical Gibbs sampler discretized at $N=512$ points are plotted on the left, and the estimated autocorrelation of each sub-chain is plotted on the right.} \label{fig:hierarchicalCorrelation}
    \end{center}
    \end{figure}

    In \citep{agapiou2013aspects,agapiou2014analysis}, the Gibbs sampling approach was analyzed theoretically on a class of hierarchical models that converge to an infinite dimensional limit.
    It was shown that the Markov chain degrades as the discretization of the forward operator converges to the continuous limit. 
    In particular, when $\lambda$ is fixed, the $\{\delta_k\}$ component of the chain makes smaller and smaller expected moves as the discretization converges, effectively slowing the exploration of the $\delta$ component of the posterior and creating highly autocorrelated samples.
    A complementary relationship between $\lambda$ and $\vect p$ was shown, where the analogous sampler that fixes $\delta$ produces a Markov chain in $\{\lambda_k\}$ that centers immediately on the true noise precision that are less and less correlated.

    We observed similar results empirically for PSF reconstruction using Gibbs sampling; that is, as the discretization limits in \eqref{eq:discretization} increase, the $\delta_k$ component of the Markov chain moves more slowly and is more highly correlated, whereas the $\{\lambda_k\}$ component of converges rapidly and looks like nearly independent samples centered on the true noise precision.  
    See \Cref{fig:hierarchicalCorrelation}.

    The dependence between the hierarchical components and the quantity of interest, specifically $\delta_k$ and $\vect p^k$, is what drives the slowing of the $\{\delta_k\}$ chain. 
    A straight-forward approach to addressing this dependence, would be to marginalize the dependence of the blocked variable $\vect \theta$ in $\vect p$. 
    That is, sample $\lambda,\delta|\vect b$ rather than $\lambda,\delta|\vect \theta^k,\vect b$.
    It turns out that the resulting sampling algorithm still provides a Markov chain that converges to the desired posterior. 
    A sampling scheme of this form has been studied by others \cite{rue2005gaussian,acosta2014markov}, and its application to a hierarchically modeled linear inverse problem is the subject of the recent work in \cite{fox2015fast}.
    An algorithm utilizing it for PSF reconstruction is given below.

    {
    \medskip 
    \captionof{algorithm}{MTC Sampler for PSF reconstruction} \label{alg:mtc}
    \hrule
    Given $\vect \theta^k = (\lambda_k,\delta_k)$ and $\vect p^k$, simulate
    \begin{enumerate}
      \item[1.] $\vect \theta^{k+1} \sim \pi(\lambda,\delta|\vect b)$
      \item[2.] $\vect p^{k+1} \sim \pi(\vect p|\lambda_{k+1},\delta_{k+1},\vect b)$
    \end{enumerate}
    \hrule
    \medskip 
    }
    
    The density $\pi(\lambda,\delta|\vect b)$ is given by computing the marginal density of $\pi(\lambda,\delta,\vect p|\vect b)$ by integrating \eqref{eq:posterior} with respect to $\vect p$.
    To see this explicitly, first write the full posterior \eqref{eq:posterior} as the Gaussian density defined in \eqref{eq:prePosteriorDensity} times terms that only involve $\lambda$ and $\delta$.
    Then, integrating with respect to $\vect p$ yields
    \begin{eqnarray}
    \pi(\lambda,\delta|\bb)
      &\propto& \lambda^{M/2}\pi(\lambda)\delta^{N/2} \pi(\delta)\exp\left(-\frac12a(\delta,\lambda)-\frac12 b(\delta,\lambda)\right),\label{marginalize2}
    \end{eqnarray}
    where
    \begin{eqnarray}
    a(\lambda,\delta)&=&\ln\left(\det(\lambda\Gb^T\Gb+\delta\Lb)\right), \label{marginalize3}\\
    b(\lambda,\delta)&=&\lambda(\bb^T\bb-\bb^T\Gb(\lambda\Gb^T\Gb+\delta\Lb)^{-1}\Gb^T\bb).\label{marginalize4}
    \end{eqnarray}
    The marginalization comes at a cost, however, as the density $\pi(\lambda,\delta|\vect b)$ is no longer one where an efficient algorithm is available, and rejection-based methods such Metropolis-Hastings must be used to simulate samples.
    Hence, the convergence in the chain is now driven by the efficiency of sampling the blocked variable $\vect \theta^k$.
    In \cite{fox2015fast}, they explore methods to accelerate this process in a deconvolution application, where because the forward operator is a convolution, Fourier based factorizations can be utilized which are not available in the application of PSF reconstruction.

    The key difference between this method and the partially collapsed sampler, that we present next, is that MTC removes the dependence between $\lambda_k$ and $\vect p^k$ which negates the predicted efficiency in sampling the $\lambda$ component when the dependence remains \cite{agapiou2013aspects}, which results in a loss of efficiency in sampling $\lambda$.
    Partial collapse integrates only the $\delta$ component with respect to $\vect p$, leaving the dependence between $\lambda_k$ and $\vect p^k$, whereas MTC keeps the variables blocked. 
    Our method requires Metropolis-Hastings only in one dimension, $\{\delta_k\}$, which requires less tuning and in the proposal.

  \subsection{Partially collapsing the Gibbs sampler for PSF reconstruction} \label{sec:partialCollapse} 

    Our sampling approach is similar to the MTC algorithm outlined in \Cref{alg:mtc}, only that we retain the dependence of $\lambda_k$ and $\vect p^k$, and as predicted by the theoretical work in \cite{agapiou2013aspects}, this dramatically improves autocorrelation of the $\{\lambda_k\}$ component of the Markov process.
    Moreover, because Metropolis-Hastings is now only on $\{\delta_k\}$, gaining efficiency in that component is more easily attained by tuning a one-dimensional random walk proposal.

    The algorithm falls into a category of samplers that are investigated in \citep{van2008partially,van2015metropolis}, where they show using spectral methods that partial collapse can improve chain convergence.
    Additionally, they show that care must be taken when modifying steps in the Gibbs sampler, since changes could result in a Markov chain whose stationary distribution is no longer the target posterior.
    In particular, they show that a partially collapsed Gibbs sampler may no longer have the same stationary distribution as the original Gibbs sampler.
    We will explicitly show that partial collapse maintains the posterior as the stationary distribution.

    {
    \medskip
    \captionof{algorithm}{Partially Collapsed Gibbs for PSF reconstruction} \label{alg:pcGibbs}
    \hrule
    Given $\lambda_k,\delta_k$ and $\vect p^k$, simulate
    \begin{enumerate}
      \item $\lambda^{k+1} \sim \pi(\lambda|\vect p^k,\vect b)$
      \item $\delta^{k+1} \sim \pi(\delta|\vect \lambda^{k+1},\vect b)$
      \item $\vect p^k \sim \pi(\vect p| \lambda^{k+1},\delta^{k+1},\vect b)$
    \end{enumerate}
    \hrule
    \medskip
    } 
    
    First note that the coupling between components in \Cref{alg:pcGibbs} is more complicated than \Cref{alg:twoStageGibbs} and \Cref{alg:mtc}, and that the algorithm is stationary with respect to $\pi(\vect p,\lambda,\delta|\vect b)$ is not obvious.
    \begin{thm}
      \Cref{alg:pcGibbs} is stationary with respect to $\pi(\vect p,\lambda,\delta|\vect b)$.
    \end{thm}
    \begin{proof}
      Denote $\pi_\bb(\cdot) = \pi(\cdot|\vect b)$ for a conditional density depending on the data $\vect b$. 
      The transition kernel associated with this algorithm is
      \begin{equation}
        K(\lambda,\delta,\vect p; \lambda',\delta',\vect p') = \pi_\bb(\vect p'|\lambda', \delta') \pi_\bb(\delta'|\lambda') \pi_\bb(\lambda'|\delta,\vect p).
      \end{equation}
      Using \eqref{eq:conditionalIndependence} to substitute $\pi_\bb(\lambda|\delta,\vect p) = \pi_\bb(\lambda|\vect p)$, the action of transition on the density $\pi_\bb(\lambda, \delta, \vect p)$ is
      \begin{align}
        \int\limits_{\R^N}\int\limits_\R\int\limits_\R & K(\lambda,\delta,\vect p; \lambda',\delta',\vect p') \pi_\bb(\lambda,\delta,\vect p)\,d\lambda d\delta d\vect p \nonumber\\
          &= \pi_\bb(\vect p'|\lambda', \delta')\pi_\bb(\delta'|\lambda') \int\limits_{\R^N}\int\limits_{\R} \pi_\bb(\lambda'|\delta,\vect p) \int\limits_\R \pi_\bb(\lambda,\delta,\vect p)\,d\lambda d\delta d\vect p\nonumber \\
          &= \pi_\bb(\vect p'|\lambda', \delta')\pi_\bb(\delta'|\lambda') \int\limits_{\R^N}\int\limits_{\R} \frac{\pi_\bb(\lambda',\delta,\vect p)}{\pi_\bb(\delta,\vect p)} \pi_\bb(\delta,\vect p) d\delta d\vect p \nonumber\\
          &= \pi_\bb(\vect p'|\lambda', \delta')\pi_\bb(\delta'|\lambda') \pi_\bb(\lambda') \nonumber\\
          &= \pi_\bb(\vect p',\lambda',\delta').
       \end{align} 
    \end{proof}


    The critical thing to note in this simple computation is that any permutation of the steps of \Cref{alg:pcGibbs} will result in a sampler that is no longer invariant with respect to $\pi(\vect p,\lambda,\delta|\vect b)$.
    This is in contrast to the Gibbs sampler, in which the steps can be permuted in any order without changing the stationary distribution of the Markov chain.

    Both \Cref{alg:mtc} and \Cref{alg:pcGibbs} have been stated in terms of sampling exactly the densities $\pi(\lambda,\delta|\vect b)$ and $\pi(\delta|\lambda,\vect b)$, where we have mentioned that these are not readily available, and we employ Metropolis-Hastings to sample them.
    The density $\vect \pi(\delta|\lambda, \vect b)$ is obtained by removing the $\lambda$-proportional terms from \eqref{marginalize2}.  

    A Metropolis-Hastings method for computing both densities requires repeated evaluations of \eqref{marginalize2}, and the computational cost is dominated by finding the determinant in \eqref{marginalize3} and computing the matrix solve in \eqref{marginalize4}. 
    For the scales relevant to PSF reconstruction, these can be accomplished by a Cholesky factorization.
    That is, we utilize the computation of an upper-triangular Cholesky factor $\vect R_{\lambda,\delta}$ which satisfies 
    \begin{equation}
      \vect R_{\lambda,\delta}^T\vect R_{\lambda,\delta} = \lambda\Gb^T\Gb + \delta\Lb, \label{eq:cholesky}
    \end{equation}
    and can be used to compute the quantities in \eqref{marginalize3} and \eqref{marginalize4} by
    \begin{align} 
      a(\delta,\lambda) &= \bb^T( \bb - \Gb \vect R_{\lambda,\delta}^{-1} \vect R_{\lambda,\delta}^{-T} \Gb^T \lambda \bb) \label{eq:term1}\\
      b(\delta,\lambda) &= 2\sum_{i=1}^N \ln|r_{i,i}(\lambda,\delta)|. \label{eq:term2}
    \end{align}
    The matrix solves in \eqref{eq:term1} are accomplished via backwards- and forwards-substitution $(O(n^2))$, and $r_{i,i}(\lambda,\delta)$ are diagonal elements of $\vect R_{\lambda,\delta}$.
    Hence, the computational cost of evaluating $\pi(\delta|\vect b,\lambda)$ is dominated by the Cholesky algorithm ($O(n^3)$), and in order to emphasize this dependence, we denote
    \begin{equation}
      c(\vect R_{\lambda,\delta}) \eqdef a(\delta,\lambda) + b(\delta,\lambda).
    \end{equation}

    A random walk in log-normal space is used for the proposal of the Metropolis-Hastings step in both the MTC and PC Gibbs algorithms.
    A full statement of each algorithm using this notation is in \Cref{alg:mtcMetropolisHastings} and \Cref{alg:pcGibbsMetropolisHastings}.

      \medskip
      \captionof{algorithm}{Metropolis-Hastings MTC Sampler for PSF Reconstruction} \label{alg:mtcMetropolisHastings}
      \hrule
      Given $\lambda_k,\delta_k,\vect p^k$, and a random walk covariance $\vect C$, simulate
      \begin{enumerate}
      \item Set $\lambda = \lambda_k$, $\delta = \delta_k$ and compute $\vect R_{\lambda,\delta}$. \\ 
        \indent For $j = 1\dots n_{mh}$
          \begin{enumerate}
            \setlength{\itemindent}{2em}
            \item[(i)] Simulate $\vect w \sim \N(0,\vect I_{2\times2})$ and set $\displaystyle \begin{bmatrix}\lambda'\\\delta'\end{bmatrix} = \exp\left(\vect C^{1/2}\vect w + \begin{bmatrix}\ln\lambda\\\ln\delta\end{bmatrix}\right)$. 
            \item[(ii)] Compute $\vect R_{\lambda',\delta'}$. 
            \item[(iii)]Simulate $u\sim U(0,1)$. If 
            $$\log u < \min\left\{ 0, c(\vect R_{\lambda',\delta'}) - c(\vect R_{\lambda,\delta}) \right\},$$ 
            then set $\lambda = \lambda ',\delta = \delta'$ and $\vect R_{\lambda,\delta} = \vect R_{\lambda',\delta'}$.
          \end{enumerate}
        \indent \indent Set $\lambda_{k+1} = \lambda$ and $\delta_{k+1}=\delta$.
      \item Simulate $\vect p^{k+1}\sim \N\left(\lambda_{k+1}{\vect R^{-T}_{\lambda_{k+1},\delta_{k+1}}} \vect R_{\lambda_{k+1},\delta_{k+1}}^{-1}\Gb^T\bb,({\vect R^T_{\lambda_{k+1},\delta_{k+1}}} \vect R_{\lambda_{k+1},\delta_{k+1}})^{-1}\right)$.
      \end{enumerate}
      \hrule 
      \medskip

    Observe that in \Cref{alg:mtcMetropolisHastings}, the Cholesky computation used for the last accepted proposal can be re-used in order to simulate $\vect p^{k+1}$.  
    In fact, because $\vect \theta^k$ does not depend on $\vect p^{k+1}$, the two components can be computed in serial; i.e.~$\vect \theta^k$ can be computed (and potentially thinned) and then samples of $\vect p^{k+1}$ can be computed, as suggested in \cite{fox2015fast}.
    This is in contrast to \Cref{alg:pcGibbsMetropolisHastings}, where an additional Cholesky factor must be computed in order to sample $\lambda_k$.
    We show in \Cref{sec:results} that the added efficiency is worth this additional computation.

    \vbox{
    \medskip 
    \captionof{algorithm}{Metropolis-Hastings PC Gibbs Sampler for PSF Reconstruction} \label{alg:pcGibbsMetropolisHastings}
    \hrule
    Given $\lambda_k,\delta_k,\vect p^k$, and $\sigma^2$, simulate
    \begin{enumerate}
    \item Simulate $\lambda_{k+1}\sim \Gamma\left(M/2+\alpha_\lambda,\frac{1}{2}\Vert\Gb\vect p^k-\bb\Vert^2+\beta_\lambda\right)$.
    \item Set $\delta = \delta_k$ and compute $\vect R_{\lambda_{k+1},\delta}$. \\ 
      \indent For $j = 1\dots n_{mh}$
        \begin{enumerate}
          \setlength{\itemindent}{2em}
          \item[(i)] Simulate $w \sim \N(0,1)$ and set $\delta' = \exp(\sigma w + \ln(\delta))$. 
          \item[(ii)] Compute $\vect R_{\lambda_{k+1},\delta'}$. 
          \item[(iii)]Simulate $u\sim U(0,1)$. If 
          $$\log u < \min\left\{ 0, c(\vect R_{\lambda_{k+1},\delta'}) - c(\vect R_{\lambda_{k+1},\delta}) \right\},$$ 
          then set $\delta = \delta'$ and $\vect R_{\lambda_{k+1},\delta} = \vect R_{\lambda_{k+1},\delta'}$.
        \end{enumerate}
      \indent \indent Set $\delta_{k+1}=\delta$.
    \item Simulate $\vect p^{k+1}\sim \N\left(\lambda_{k+1}{\vect R^{-T}_{\lambda_{k+1},\delta_{k+1}}} \vect R_{\lambda_{k+1},\delta_{k+1}}^{-1}\Gb^T\bb,({\vect R^T_{\lambda_{k+1},\delta_{k+1}}} \vect R_{\lambda_{k+1},\delta_{k+1}})^{-1}\right)$.
    \end{enumerate}
    \hrule 
    \medskip 
    }

    The theoretical justification for the use of Metropolis-Hastings as a sub-step in samplers can be found in  \citep{robert2013monte,gamerman2006markov,kaipo2005}.
    Note that each proposal step requires a computationally expensive Cholesky solve.
    The authors in \citep[Chapter 10.3]{robert2013monte} suggest that for Metropolis-within-Gibbs, additional proposals ($n_{\rm MN}>1$) may not be worth the computational cost, while others have suggested more sub-steps \citep[Section 6.4.2]{gamerman2006markov,van2015metropolis} to improve convergence.
    The situation likely depends on the problem, and due to the lack of objective criteria, we investigate empirical evidence that suggests that in the case of PSF reconstruction, more than one step can improve convergence.

\section{Results} \label{sec:results}

  In this section, each of the three methods described are used to analyze synthetically generated and real data from a diagnostic radiographic imaging systems in operation at the Nevada National Security Site.
  We first establish the metrics by which we compare them in order to fairly compare the algorithms.
  In particular, we briefly describe a statistical method for determining whether the chain has reached stationarity and passed the so-called burn-in stage and how efficiently the chains explore the invariant density by estimating the stationary autocorrelation.
  Our measure of efficiency also takes into account computational effort, and we show that PC Gibbs for hierarchical sampling performs significantly better than standard Gibbs sampling and at least as well as the recently developed MTC sampler.

  \subsection{Statistical measures of convergence.} \label{subsec:convergeceStatistics}
    The stationarity of the partially collapsed Gibbs sampler guarantees that Monte Carlo realizations of the Markov process converge in distribution to realizations from $\pi(\vect p, \lambda,\delta|\bb)$, but this asymptotic result does not address the practical fact that only a finite number of simulations can be computed. 
    Two aspects of convergence are addressed in this section.
    First, the initial samples must converge, or \emph{burn-in}, to the desired stationary distribution of the Markov chain, and second, since the process produces identically distributed but \emph{dependent} samples, how effectively uncorrelated the samples are determines how well they characterize the stationary distribution.
    In this section, we give a brief overview of two statistical estimators that address these two aspects.
    Both estimators inform how long to run the MCMC algorithm to effectively analyze the chain as a robust sample from $\pi(\vect p,\lambda,\delta|\bb)$.

    The first issue is concerned with how close the Markov chain is to the target invariant density.
    In practice, the Markov chain is initialized with simulations that are not from the target density, $\pi(\vect p, \lambda,\delta|\bb)$, and \citep{geweke1991evaluating} provides statistically motivated approach that uses a statistical test to evaluate the test hypothesis that the joint mean value of an early section of the Markov chain is equal to that of a latter portion.
  Formally, for a given univariate component of a stochastic process, $\{X^1,\dots, X^N\}$, let $N_m$ denote the $m$th percentile of $N$, $\mu_m$ to be the mean of $\{X^1,\dots, X^{N_m}\}$ and $\mu_{m'}$ the mean of $\{X^{N_{m'}+1},\dots, X^N\}$.
  Following \citep{geweke1991evaluating}, we choose the 10\textsuperscript{th} and 50\textsuperscript{th} percentiles to establish the estimators for $\mu_{10}$ and $\mu_{50'}$, which  are
  \begin{equation}
    \bar X_{10}=\frac{1}{N_{10}}\sum_{k=1}^{N_{10}}X^k,\quad{\rm and}\quad\bar X_{50'}=\frac{1}{N-N_{50'}}\sum_{k=N_{50}+1}^N X^k.
  \end{equation}
  For the test $H_0:\mu_{10} = \mu_{50'}$, \citep{geweke1991evaluating} shows the corresponding convergence diagnostic test statistic satisfies
  \begin{equation} 
    R_{\rm Geweke}\eqdef\frac{\bar X_{10}-\bar X_{50'}}{\sqrt{\hat S_{10}(0)/N_{10}+\hat S_{50'}(0)/N_{50}}}\stackrel{d}{\longrightarrow}\N(0,1),\quad{\rm as}\quad N\rightarrow\infty, \label{eq:geweke}
  \end{equation}
  where $\hat S_{10}(0)$ and $\hat S_{50'}(0)$ denote consistent spectral density estimates for the variances of $\{X^1,\dots, X^{N_{10}}\}$ and $\{X^{N_{51}},\dots, X^N\}$, respectively. 
  These can be estimated via a periodogram estimator, and in our results, we use a Danielle window of width $2\pi/(0.3p^{1/2})$ as recommended by \citep{geweke1991evaluating}.
  The test provides a method for evaluating a portion of the realizations of the Markov chain that are suitable for a rigorous analysis.
  For the results in this paper, each algorithm was run for a fixed number of iterations, and the last half of the simulations were tested. 
  The process is then assumed to be in stationarity if the test provides no statistical evidence for a difference in the quantile means.

  The second estimator we establish measures how efficiently the stationary Markov chain characterizes the posterior density.
  That is, after identifying the burn-in portion of the chain, successive simulations may be highly correlated and result in an excessively slow exploration of the target density. 
  Improving this aspect of convergence is the primary motivation for partial collapse.
  Following \citep{sokal1997monte}, we use the notion of integrated autocorrelation time to quantify how much the Monte Carlo samples have explored the target density relative to a hypothetical independent sample.
  Summarizing that work, suppose that $\{X_1,X_2,\dots\}$ is an identically distributed correlated stochastic process with individual variance $\sigma^2$, then the Monte Carlo error for the estimator $\bar X_N=\frac1N\sum_{k=1}^NX^i$ can be divided into a contribution from inherent variance in $X_j$, and covariance between $X_i$ and $X_j$ for $j\not=i$; i.e.
  \begin{align}
  {\rm Var}(\bar X_N)
    &=\frac{\sigma^2}{N}\left(1+2\sum_{k=1}^{N-1}\left(1-\frac{k}{N}\right)\frac{{\rm Cov}(X^1,X^{1+k})}{\sigma^2}\right).
    \label{eq:monteCarloError}
  \end{align}
  The autocorrelation function at lag $k$ of the process is $\rho(k) \eqdef \frac{\mathrm{Cov}(X^1, X^{|k|})}{\sigma^2}$, and so for large $N$, the Monte Carlo error can be approximated with
  \begin{equation} 
    {\rm Var}(\bar X_N) \approx \frac{\sigma^2}{N}\sum_{k=-\infty}^\infty \rho(k) \eqdef \frac{\sigma^2}N \tau_{\rm int}.
    \label{eq:monteCarloErrorEstimator}
  \end{equation}
  The approximation is based on the assumption that the autocorrelation lag of the process dies off fast enough so that $k/N$ does not contribute to \eqref{eq:monteCarloError}.

  Since $\sigma^2/N$ would be the variance of the Monte Carlo estimator had $\{X_1,\dots,X_N\}$ been uncorrelated, we think conceptually of the parameter $\tau_{int}$ as the equivalent number of Markov chain simulations required to obtain an effectively independent sample from the target density (in terms of Monte Carlo error of sample mean estimation).
  This analogy motivates what is sometimes called the essential sample size of the chain 
  \begin{equation}
    N_{\rm ESS} \eqdef N/\tau_{\rm int}. \label{eq:ess}
  \end{equation}
  To estimate these parameters, \citep{sokal1997monte} gives the following unbiased estimator for the normalized autocorrelation function,
  \begin{equation}
    \hat\tau_{\rm int}=\sum_{k=-\bar N}^{\bar N} \hat{\rho}(k),
  \end{equation}
  where $\bar N< N-1$ is some window length, and $\hat \rho(k)$ is the empirical normalized covariance estimator over that interval.
  That is,
  \begin{align}
  \hat\rho(k)&\eqdef \hat C(k)/\hat C(0), \quad\text{where}\quad
  \hat C(k)=\frac{1}{N-k}\sum_{i=1}^{N-k} (X_i-\bar{X_N})(X_{i+k}-\bar{X_k}).
  \end{align}
  The choice we use suggested by \cite{sokal1997monte} for the window size is the smallest integer such that $\bar N\geq 3 \hat\tau_{\rm int}$.  
  Finally, our estimate the essential sample size, denoted ESS, is given by substituting the estimator $\hat\tau_{\rm int}$ for $\tau_{\rm int}$ in \eqref{eq:ess}.
  
  The ESS estimate can be used in a couple of ways. 
  A standard approach, when samples are relatively cheap to compute, is to do as follows: (1) compute a very long MCMC chain (we choose $10^4$ below); (2) remove the first half of the chain as burn-in and verify using Geweke's test that the second half of the chain is in equilibrium; and (3) estimate the number of effectively independent samples in the second half of the chain using $\hat{K}_{\rm ESS}$. 
  In cases in which each sample is expensive to compute, however, there is incentive to make the chain as short as possible. 
  In such instances, both chain convergence and autocorrelation can be monitored online, so that a minimal number of samples are discarded in the burn-in stage, and also so that $\hat{K}_{\rm ESS}$ is not larger than it needs to be in order to perform the desired uncertainty analysis.
  
  The ESS is not the complete answer to the efficiency of the algorithm.
  Highly uncorrelated chains for $\delta$ can be achieved in the MTC and PC Gibbs algorithm by increasing the number of inner Metropolis-Hastings steps $n_{\rm MH}$, however, the addition of each step increases the number of expensive matrix factorizations by a factor of the chain length.
  To take the extra computational effort into account, we use the number of Cholesky factorizations as a metric for computational effort since this computation dominates the computational time per MCMC iteration.
  That is, we use the number of Cholesky solves divided by the ESS, which we interpret as the computational effort to obtain an equivalent uncorrelated sample.
  If we assume that a chain thinned according to estimated integrated autocorrelation is equivalent to an uncorrelated sample, this measure says how computationally costly it is to obtain each sample.

  \subsection{Synthetic examples of PSF reconstruction}\label{subsec:synthethic}

  We first establish the efficacy of our approach on a simulated example where the true profile is explicitly known, and the data is artificially corrupted with simulated noise.
  To simulate synthetic data, we reconstruct the radial profile of a two-dimensional Gaussian kernel
  \begin{equation} \label{eq:syntheticx}
    x(r) = (2\pi\sigma^2)^{-1} e^{\frac{-r^2}{2\sigma^2}},
  \end{equation}
  where $\sigma = \frac1{15}$ is chosen so that the effective width of the kernel is about 20\% of the image width when scaled to $[-1,1]$.
  Observe that in the case of a two-dimensional Gaussian, the action of the forward operator in (\ref{eq:psfForwardModel}) is the scaled error function
  \begin{equation} \label{eq:syntheticb}
    b(s) = \frac 1{\sqrt{2\pi}\sigma} \int_{-\infty}^s e^{-\frac{s'^2}{2\sigma^2}}\,ds',
  \end{equation}
  which can be numerically calculated with very high precision.
  Gaussian measurement error with noise strength that is 2\% of the strength of the signal is synthetically generated and added to $b(s)$.

  All three MCMC algorithms were run with a chain length of $N=10^4$, and the last $5\times10^3$ simulations were tested with the Geweke statistic for stationarity.
  For each algorithm, the resulting $p$-values were all greater than $0.9$, hence we use the last $5\times 10^3$ simulations as burned-in MCMC samples from the posterior density.

  To estimate the PSF and the hierarchical parameters, we use the sample mean of the burned-in samples.
  The true PSF falls well within the distribution of MCMC samples, and the mean MCMC estimator for the PSF matches the truth quite well; see the left panel of \Cref{fig:syntheticPsfRecon}.
  Note that the most uncertain region of the reconstruction are the initial discretization points corresponding to the height of the PSF. 
  Since the simulated data has a known solution, if we interpret the problem variationally with the Tikhonov regularization parameter $\delta/\lambda$, we can characterize a ``best'' regularization by minimizing the $L^2$ norm of the residual with respect to the Tikhonov regularization parameter.
  In right panel of \Cref{fig:syntheticPsfRecon}, a plot of the log $L^2$ norm of the residuals versus the Tikhonov regularization parameter are given with the MCMC estimate $\hat \delta/\hat \lambda$ indicated by an asterisk. 
  Note that the MCMC estimator nearly falls on the minimum of the curve.
  \begin{figure}[htbp!]
    \begin{center}
      \includegraphics[width=.45\textwidth]{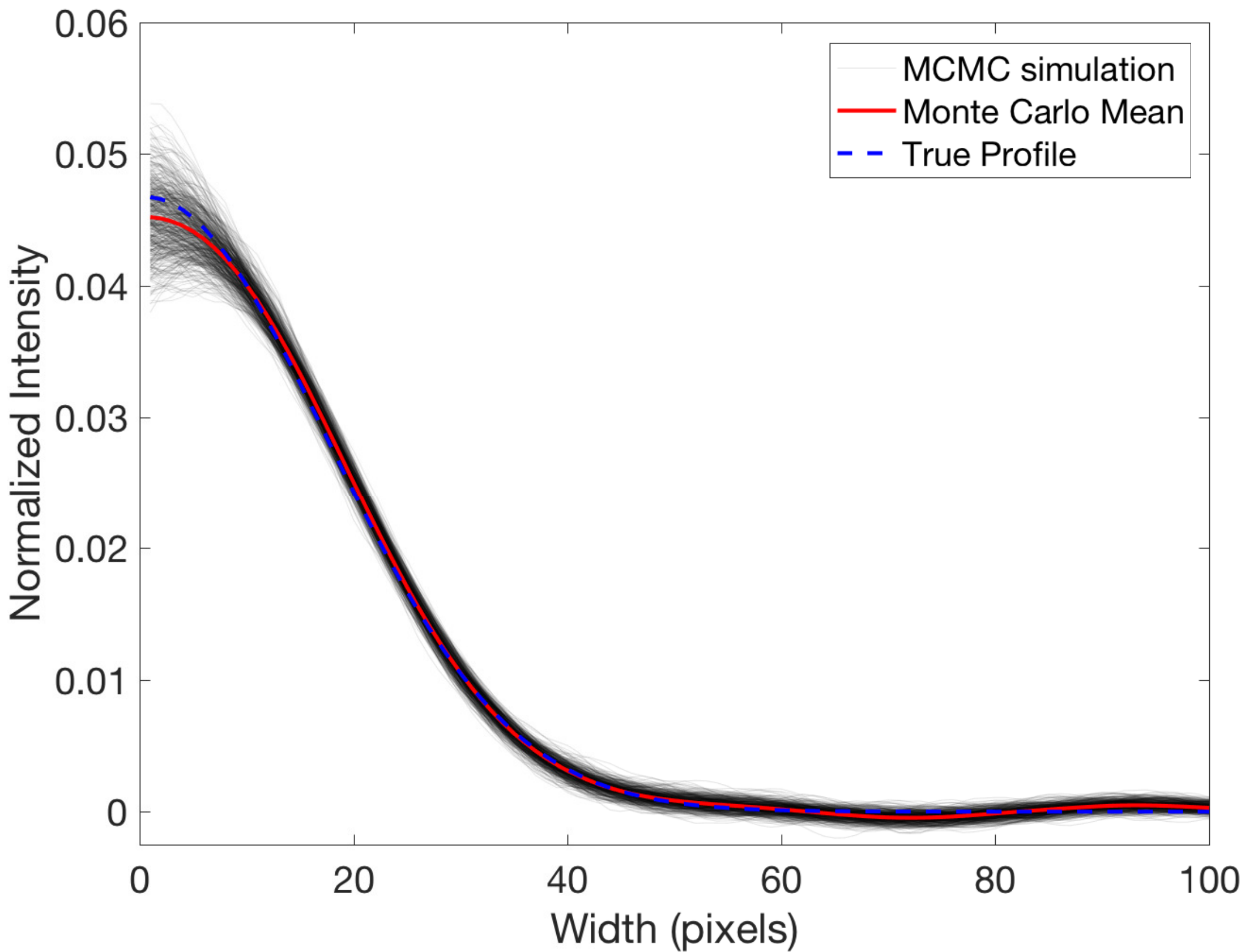}
      \includegraphics[width=.45\textwidth]{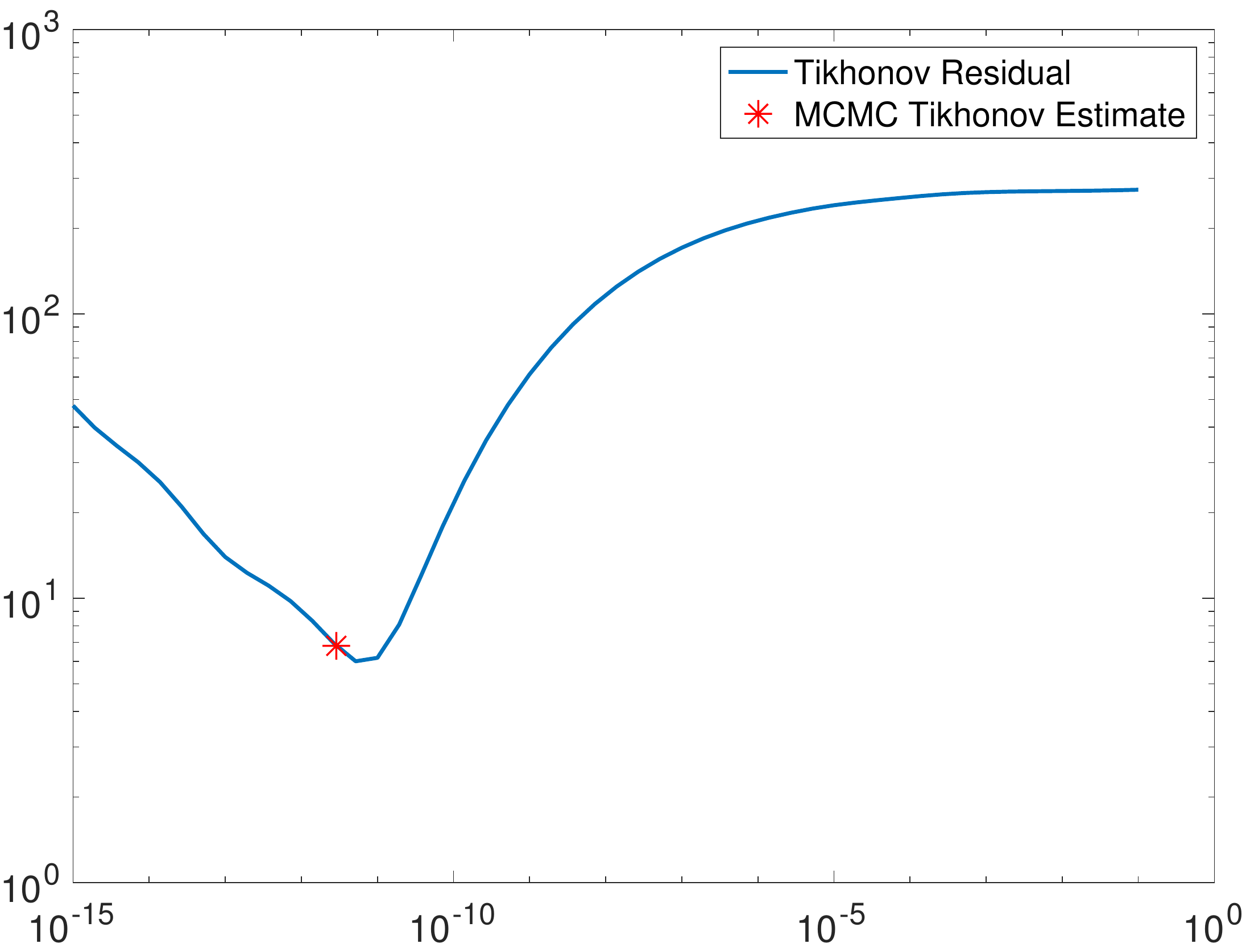}
    \caption{(left) The true solution, the mean MCMC estimate, and $10^3$ Monte Carlo samples are plotted together; (right) the sum of squared residuals of the least square solution verses the corresponding Tikhonov parameter with the MCMC estimate $\hat \delta/\hat \lambda$. } \label{fig:syntheticPsfRecon}
    \end{center}
  \end{figure}

  In the left panel of \Cref{fig:increaseNmh}, note that increases in the Metropolis-Hastings sub-steps of MTC generally decreases the efficiency of the sampler.  
  This is because of the discussion in \Cref{sec:mcmcAlgorithms}, where it was shown that the sampling of $\vect \theta^k$ does not depend on $\vect p^k$ and only on the previous $\vect \theta^{k-1}$.  
  Since the efficiency of sampling $\vect \theta^k$ depends only on the Metropolis-Hastings algorithm with no influence from $\vect p^k$, the difference is merely how the number Cholesky factorizations are accounted for per iteration of the algorithm.
  Said another way, increasing $n_{\rm MH}$ of MTC is equivalent to simulating a chain of $\vect \theta^k$ of length $M \cdot n_{\rm MH}$, with $\vect p^k$ simulated in a chain of length $M$.
  In the case of the PC Gibbs sampler on the other hand (right panel of \Cref{fig:increaseNmh}), extra Metropolis-Hastings steps appear to increase the efficiency of the sampler.
  This is because only $\delta$ has been marginalized, and even though $\lambda_k$ and $\delta_k$ can still be blocked into $\vect \theta^k$, it depends on $\vect p^k$ through $\lambda_k$.  
  This means that the PC Gibbs transition kernel improves with better samples of $\delta_k$ that increased Metropolis-Hastings steps provide.
  \begin{figure}[htbp!]
  \begin{center}
    \includegraphics[width=.45\textwidth]{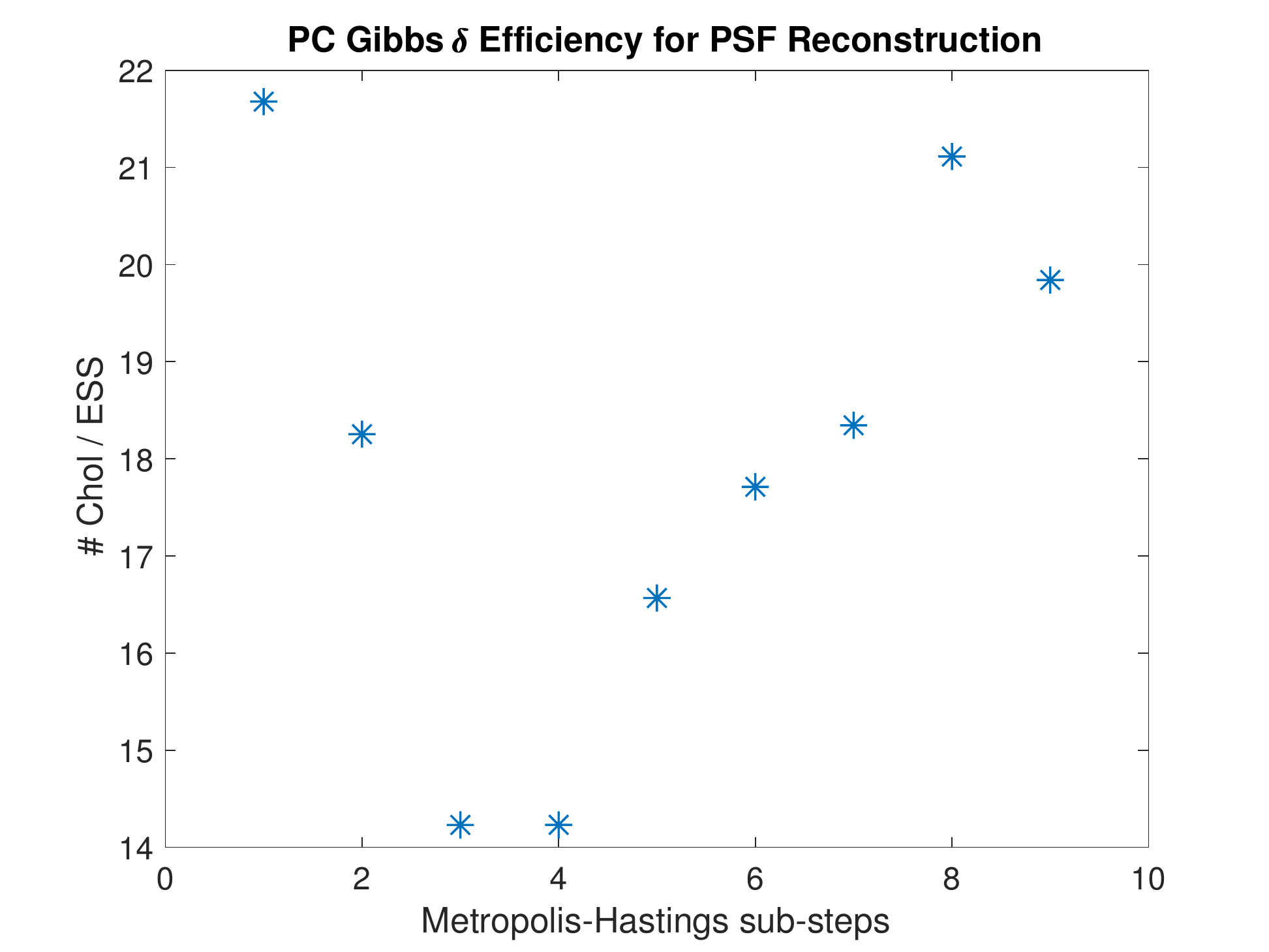}
    \includegraphics[width=.45\textwidth]{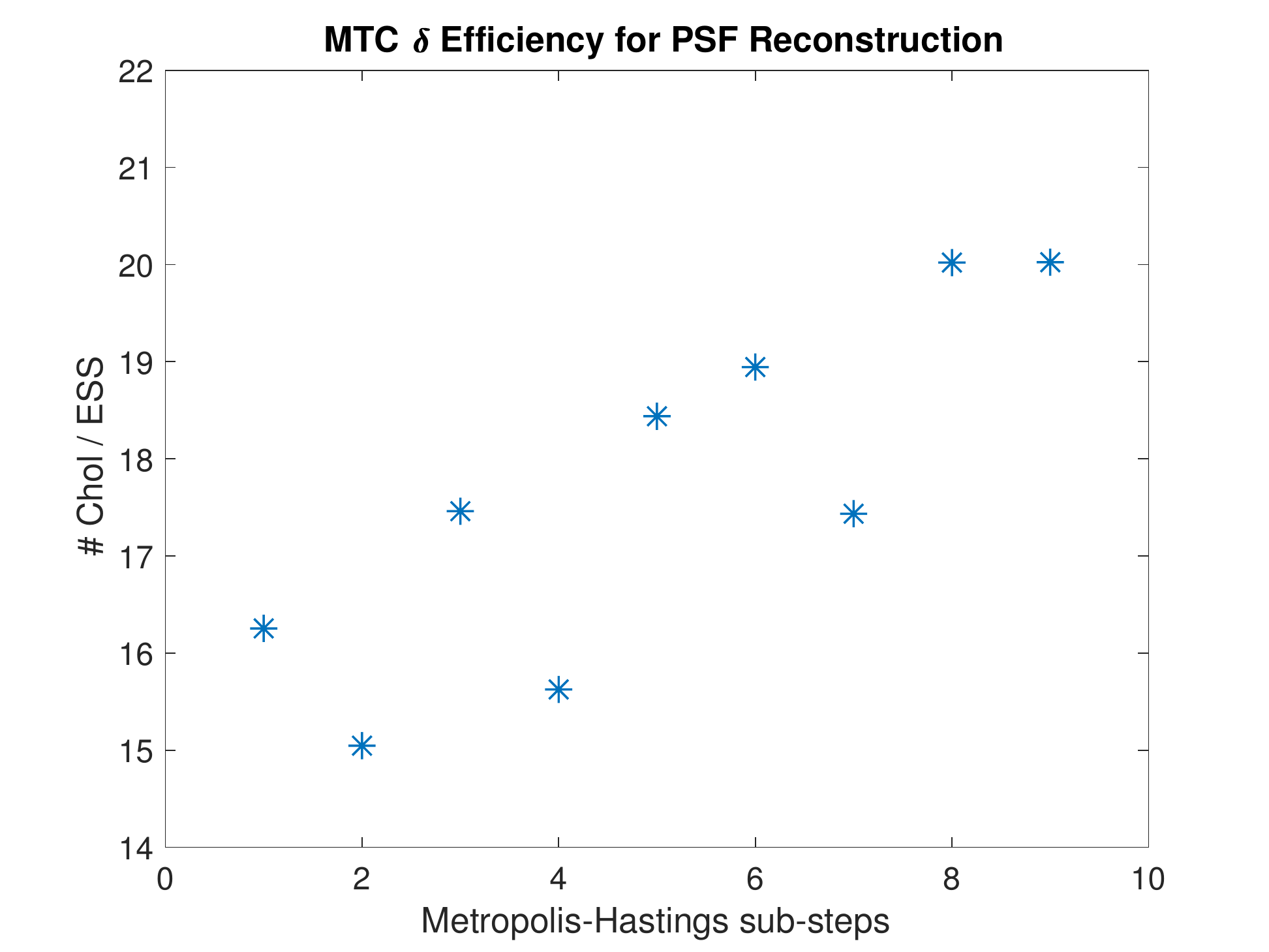} 
  \end{center}
    \caption{The efficiency (\#Chol/ESS) of the $\delta_k$ component verses the number of Metropolis-Hastings substeps for the PC Gibbs algorithm on edge data with $N=797$.}
    \label{fig:increaseNmh}
  \end{figure}

  The efficiency statistics for each algorithm are given in \Cref{tab:syntheticEfficiency}.
  In order to give a common basis for comparison, the proposal parameters for each Metropolis-Hastings random walk in MTC and PC Gibbs were derived from the burned-in posterior samples of the hierarchical Gibbs sampler.
  Specifically, two times the empirical covariance of the burned-in $\vect \theta^k$ from a realization of the hierarchical Gibbs sampler was used as the proposal covariance for MTC.
  Two times the variance of $\delta_k$ from the same realization was used as the proposal variance for PC Gibbs. 
  These choices result in acceptance rates near 0.3 for MTC and 0.45 for PC Gibbs.

  First observe that each sampler generally agrees in terms of the MCMC means generated.  
  As predicted, the $\lambda_k$ sub-chains are all sampled very efficiently in Gibbs and PC Gibbs, while its efficiency is driven by Metropolis-Hastings in MTC.  
  Moreover, the efficiency of the PC Gibbs algorithm with $n_{\rm MH}=4$ is estimated to be slightly more efficient than MTC.
  We remark that due to the variability in the proposal tuning, small differences in efficiency are likely not definitive, but this result serves as evidence that the efficiencies of PC Gibbs and MTC are roughly equivalent.  
  The advantage of PC Gibbs is that its Metropolis-Hastings proposal need only be tuned in one dimension.  

  \begin{table}[htbp!]
  \begin{center}
  \caption{ Statistical diagnostics for the $\lambda$ and $\delta$ chains associated with the synthetic PSF reconstruction problem. 
              The first two columns are the post-burn-in chain means of $\lambda$ and $\delta$. 
              The Metropolis-Hastings proposal acceptance rate is given in the third column and the estimated efficiency of the $\lambda_k$ and $\delta_k$ components are given in the fourth and fifth columns.} \label{tab:syntheticEfficiency}
    \begin{tabular}{l|cccccc}
      \hline
       Algorithm  & $\hat{\lambda}_{\rm MCMC}$& $\hat{\delta}_{\rm MCMC}$  & MCMC       & $\delta$ &  $\lambda$ \\ 
                  & $(\times 10^{4})$         & ($\times 10^{-7}$)         & Acc. Rate  & \#Chol/ESS  & \#Chol/ESS \\ \hline
        Gibbs     & 1.162     & 1.125     &       1.0 &    58.181&        1.1 \\
          MTC     & 1.160     & 0.977     &     0.301 &    16.251&       17.5 \\ \hspace{.2in} $n_{mh}= 1$ & & & & & \\
         PC Gibbs & 1.162     & 1.002     &     0.446 &    21.673&        1.0 \\ \hspace{.2in} $n_{mh}= 1$ & & & & & \\
         PC Gibbs & 1.160     & 1.006     &     0.475 &    14.228&        1.1 \\ \hspace{.2in} $n_{mh}= 4$ & & & & & \\ \hline 
    \end{tabular}
  \end{center}
  \end{table}
  \begin{figure}[htbp!]
  \begin{center}
    \includegraphics[width=.45\textwidth]{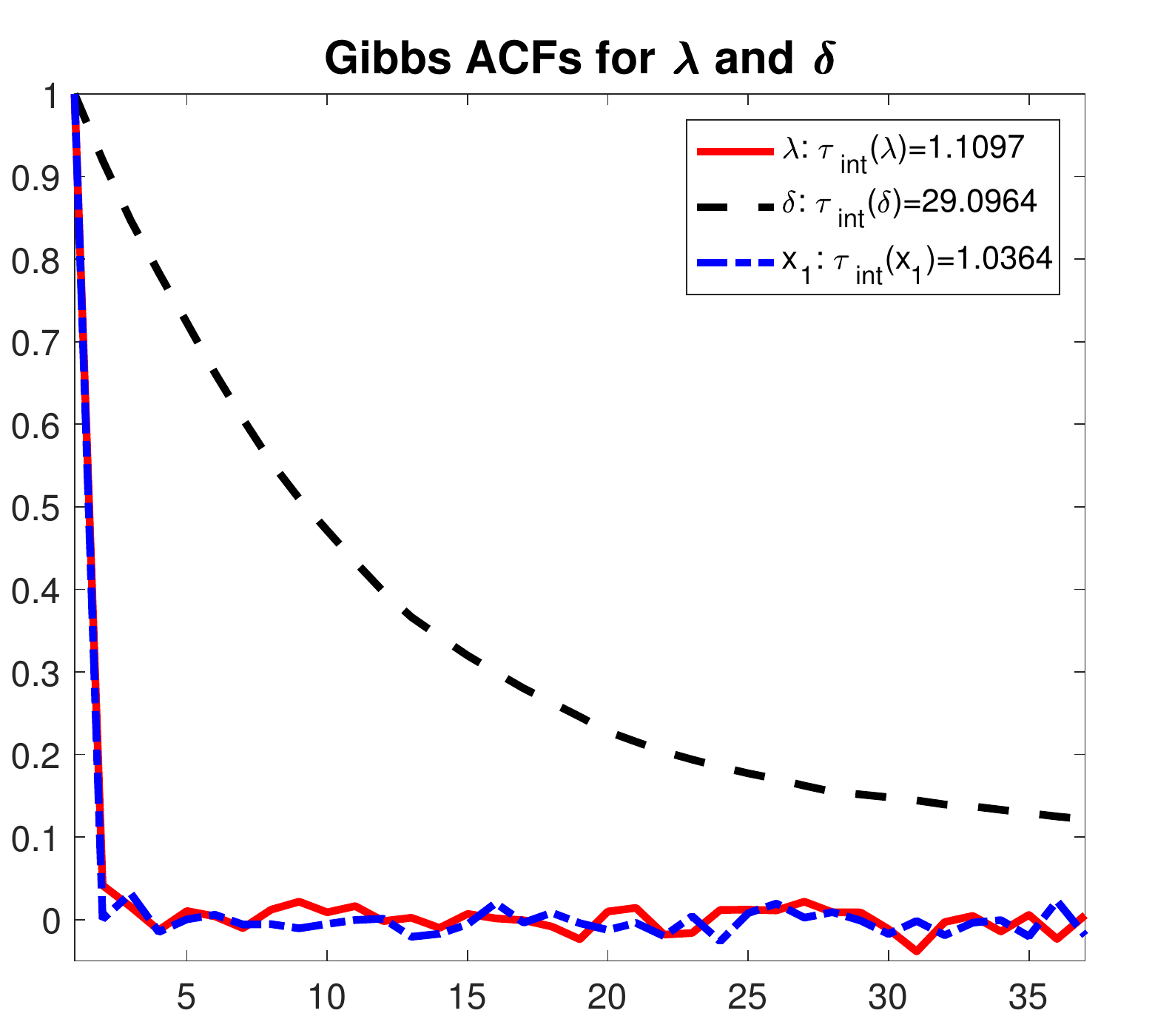}    \includegraphics[width=.45\textwidth]{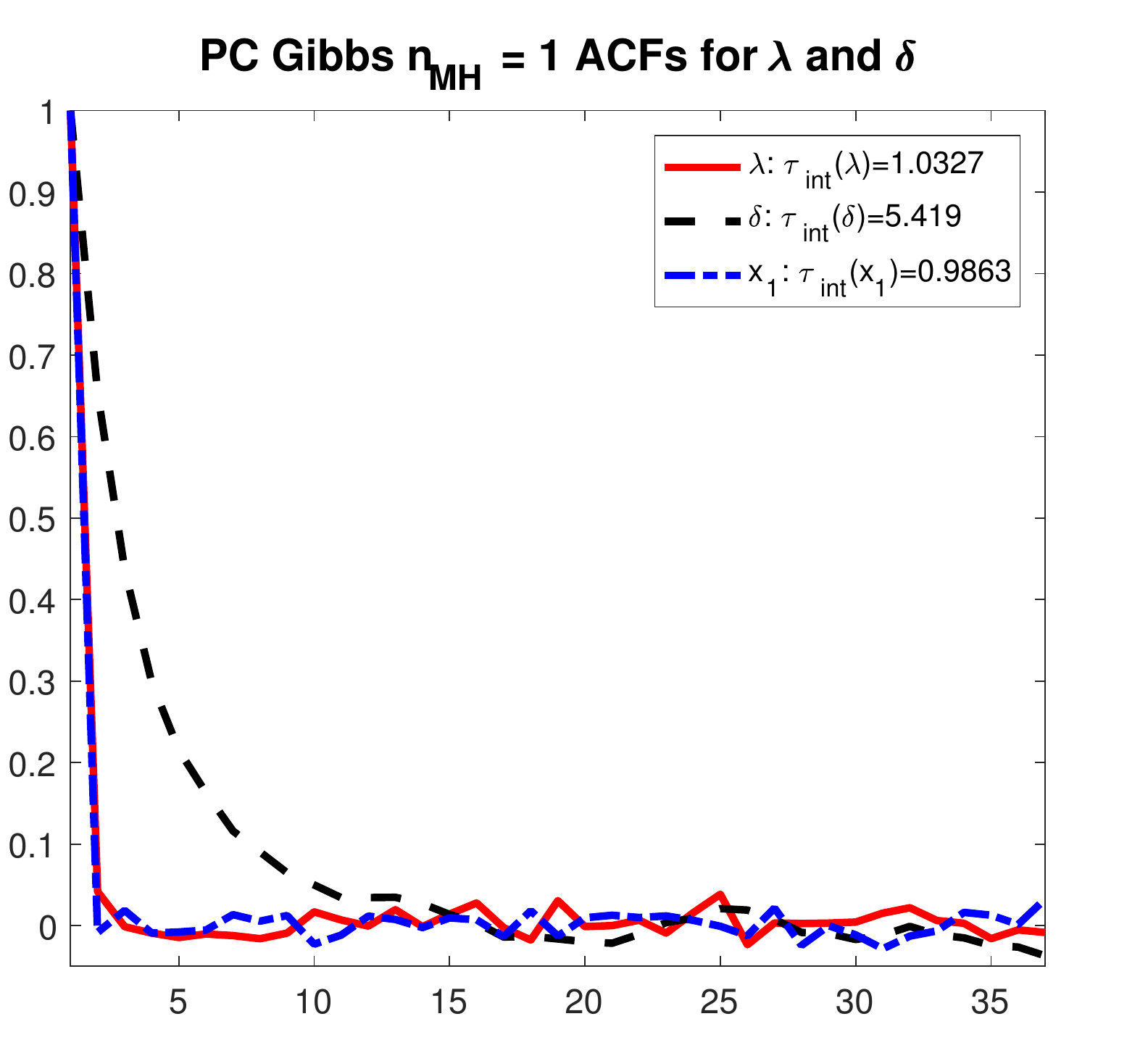}\\\vspace{2em}
    \includegraphics[width=.45\textwidth]{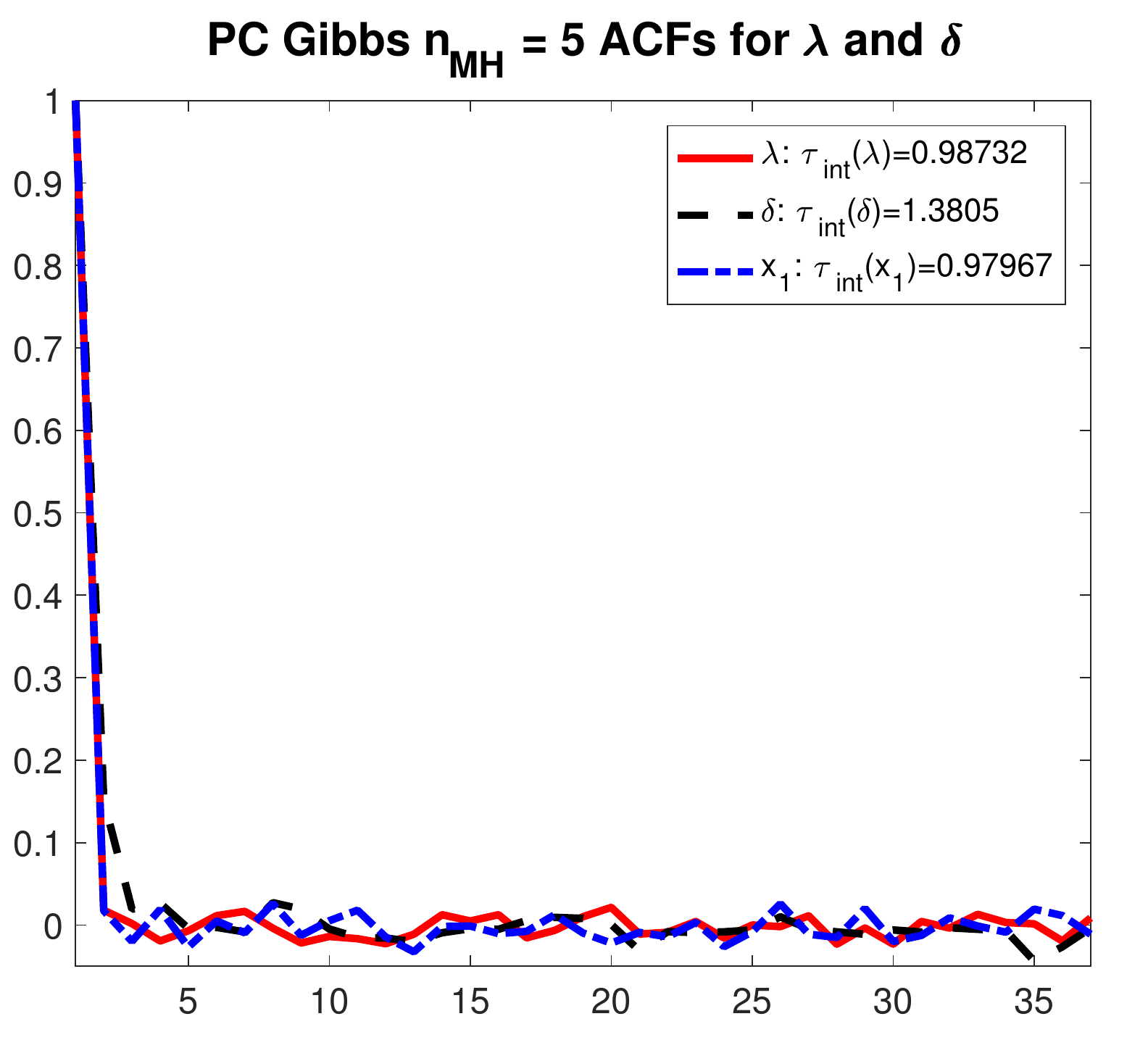}\includegraphics[width=.45\textwidth]{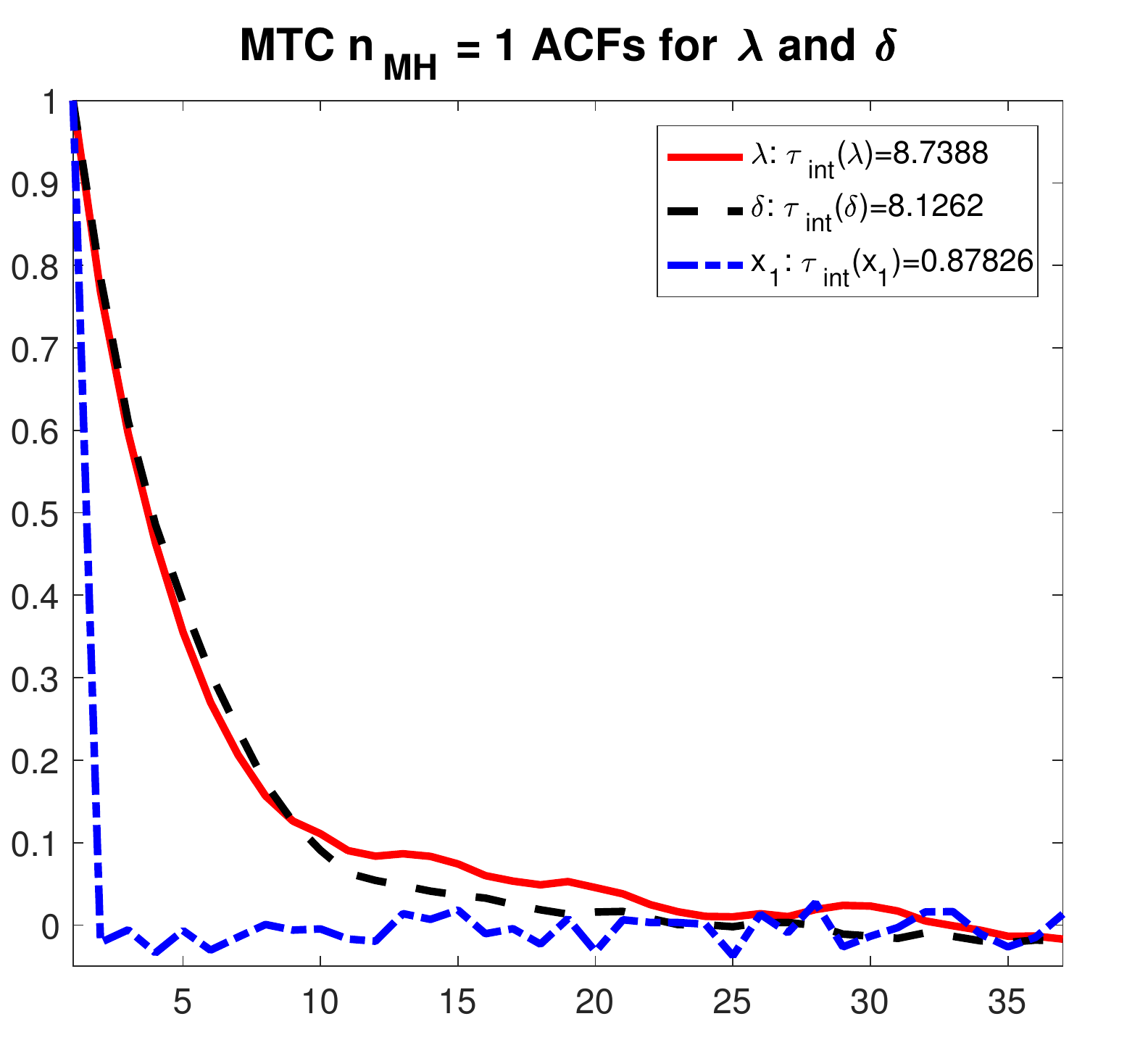}
  \end{center}
  \caption{Autocorrelation plots for PSF reconstruction for synthetic data of the sub-chains for $\lambda$, $\delta$ and the central discretization point of $\vect p$: in the upper-left are the ACF for Markov chains of $\lambda$, $\delta$ and central pixel of the radial profile for the Gibbs sampler; on the upper-right are the plots for the PC Gibbs sampler with 1 inner MH step; on the lower-left are plots for the PC Gibbs sampler with 5 inner MH steps; and in the lower-right are plots for the MTC sampler.} \label{fig:syntheticPsfACF}
  \end{figure}

  \subsection{PSF reconstruction with measured radiographic data}\label{subsec:realData}

    Next we reconstruct the point spread function of a high energy X-ray imaging system at the U.S.~Department of Energy's Nevada National Security Site. 
    The real edge data is shown in \Cref{fig:CygnusPsfRecon} (upper left) along with a horizontal cross-section across the edge (upper right). The mean MCMC reconstruction is shown in \Cref{fig:CygnusPsfRecon} (lower left), along with the 10\%, 25\%, 50\%, 70\%, and 90\% quantiles of the chain $\bm{x}^{k}$.
    We estimated the PSF at grid points using the chain-wise mean after burn-in, $\hat {\vect p }= \frac 2M \sum_{k=M/2+1}^M {\vect p}^{k}$.
    Since the true PSF is unknown, we evaluate the accuracy of the estimation by its discrepancy; i.e. we compared forward mapping of the estimate $\Gb {\hat {\vect p}}$ with the given data $\bm b$.  This is shown in both linear and logarithmic scales in \Cref{fig:CygnusPsfRecon} (lower right). In both cases the discrepancy is quite low, except at very low intensities where the data is dominated by the noise, which can be seen in the logarithmic scale.
    Observe that the chain efficiency statistics in the third through fifth columns of \Cref{tab:CygnusPsfRecon} are similar to those derived on the synthetic example.

  \begin{figure}[ht]
    \begin{center}
    \includegraphics[height=2in]{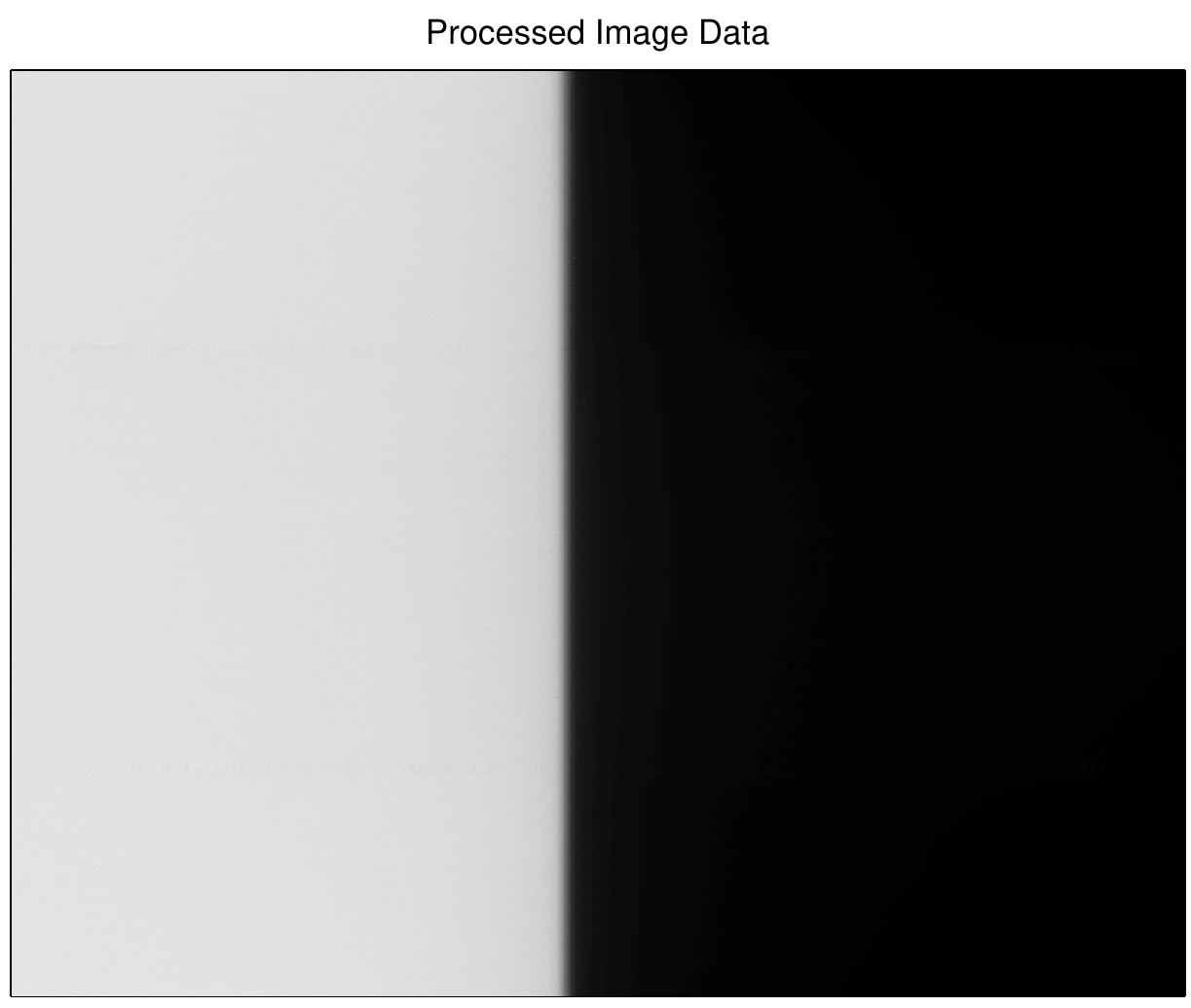}   \hspace{0em}\includegraphics[height=2in]{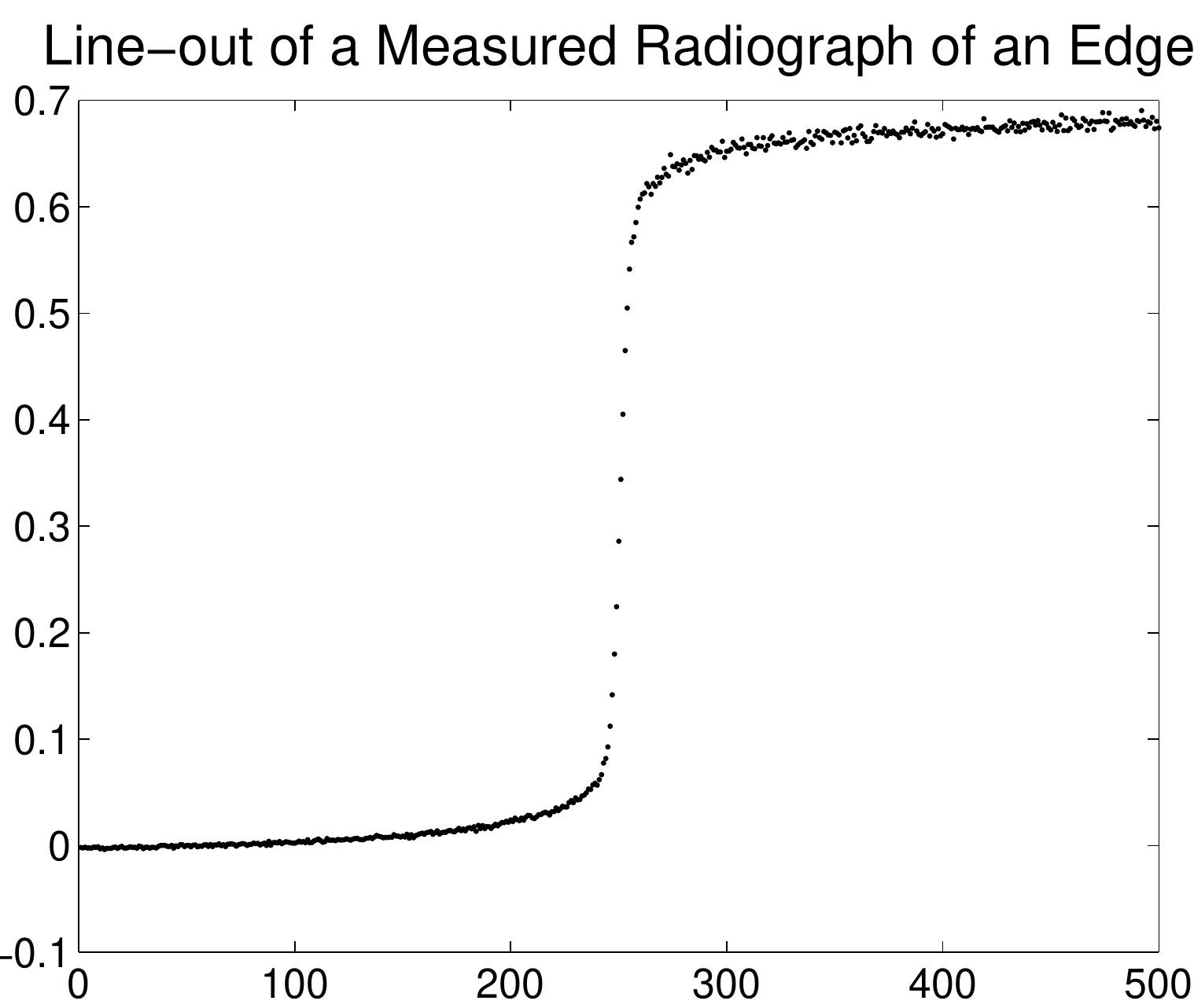}\\\vspace{2em}
    \includegraphics[width=.5\textwidth]{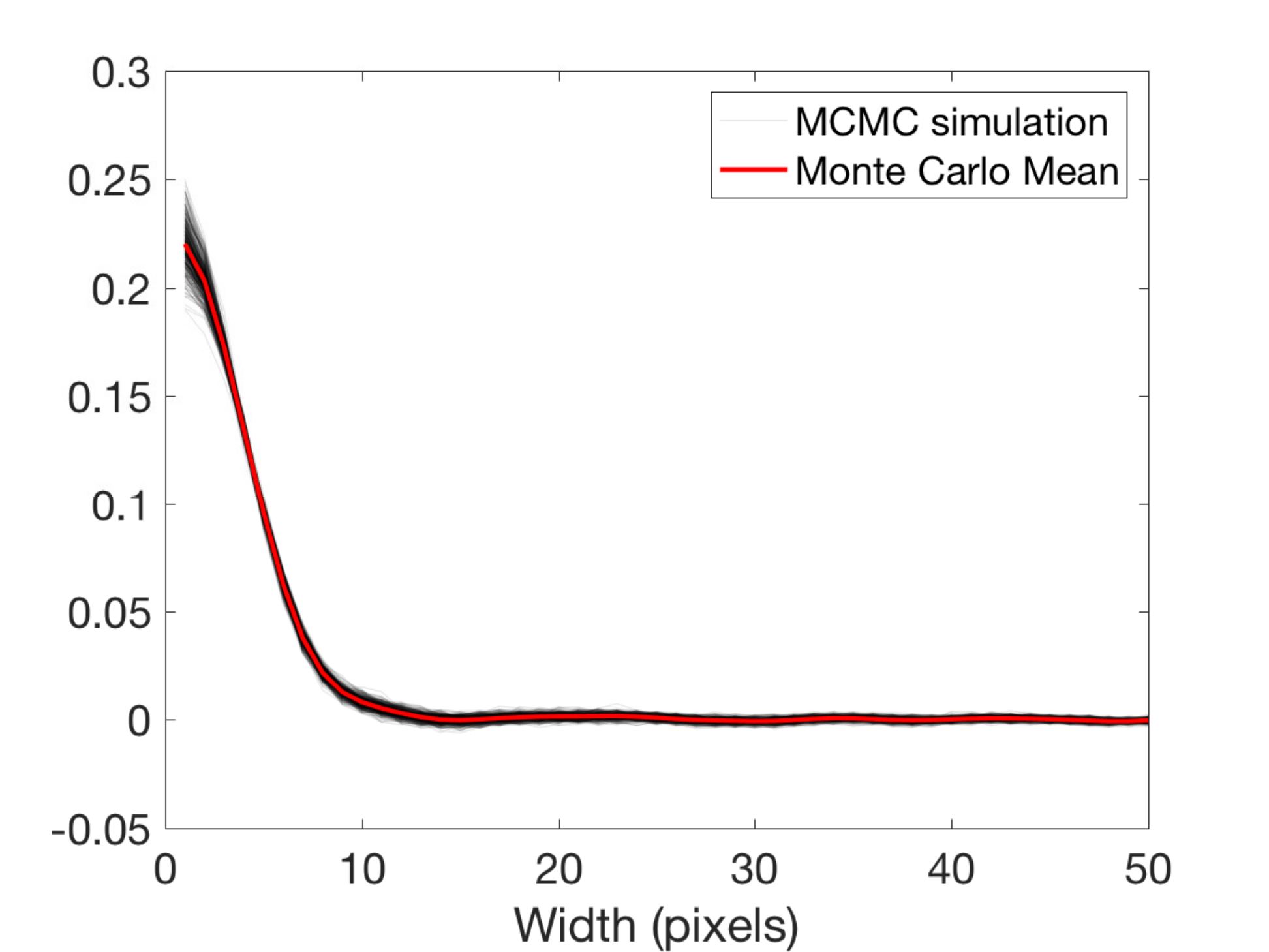}\hspace{0em}\includegraphics[width=.5\textwidth]{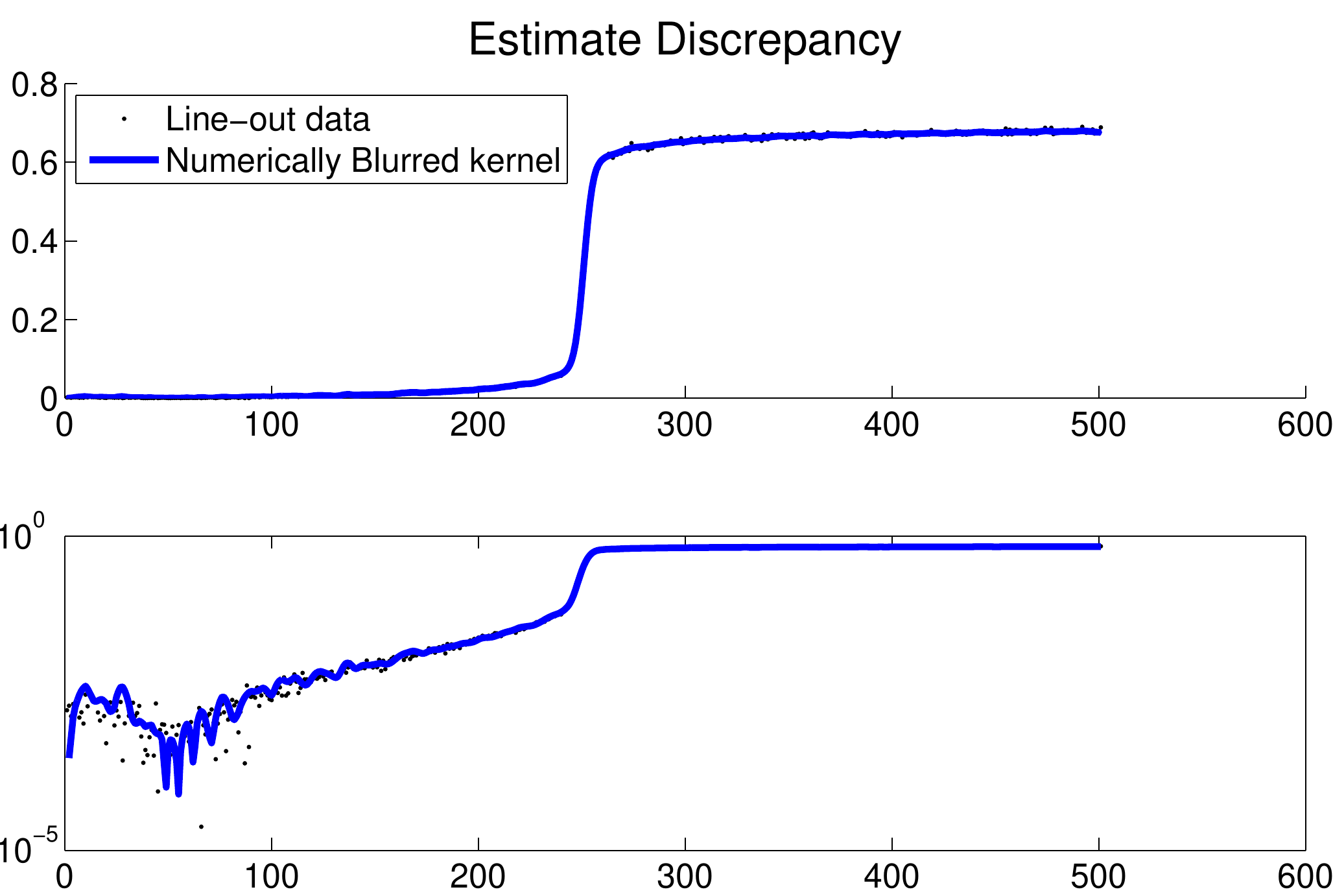}
    \caption{
    PSF reconstructions for radiographic data: in the upper left corner are the radiographic image data;
    in the upper right corner is a line-out taken from the image data;
    in the lower left corner are the central 10\%, 25\%, 50\%, 70\%, and 90\% quantiles of the posterior reconstruction of $\bm x$ for each pixel;
    in the lower right corner are plots of the forward mapped discrepancy of the post burn-in chain mean.
    } \label{fig:CygnusPsfRecon}
    \end{center}
  \end{figure}

  Similar to the synthetic data, the PC Gibbs algorithm with $n_{mh}=4$ and MTC perform with roughly equivalent overall efficiency, however, with $\lambda_k$ being the least efficient component of MTC.

  \begin{table}[htbp!]
  \begin{center}
    \caption{ Statistical diagnostics for the $\lambda$ and $\delta$ chains associated with the PSF reconstruction problem on real radiographic data. 
              The first two columns are the post-burn-in chain means of $\lambda$ and $\delta$. 
              The Metropolis-Hastings proposal acceptance rate is given in the third column and the estimated efficiency of the $\lambda_k$ and $\delta_k$ components are given in the fourth and fifth columns.} \label{tab:CygnusPsfRecon}
    \begin{tabular}{l|ccccc}
       Algorithm  & $\hat{\lambda}_{\rm MCMC}$& $\hat{\delta}_{\rm MCMC}$  & MCMC       & $\delta$ &  $\lambda$ \\ 
                  & $(\times 10^{4})$         & ($\times 10^{-10}$)         & Acc. Rate  & \#Chol/ESS  & \#Chol/ESS \\ \hline
      \hline
         Gibbs  & 9.148     & 1.205     &     1.000 &    36.705&        1.4 \\
           MTC  & 9.117     & 1.141     &     0.308 &    14.524&       16.4 \\ \hspace{.2in} $n_{mh}= 1$ & & & & & \\
       PC Gibbs & 9.148     & 1.152     &     0.442 &    21.092&        1.3 \\ \hspace{.2in} $n_{mh}= 1$ & & & & & \\
       PC Gibbs & 9.149     & 1.150     &     0.452 &    15.535&        1.5 \\ \hspace{.2in} $n_{mh}= 5$ & & & & & \\ \hline 
      \hline
    \end{tabular}
  \end{center}
  \end{table}

\section{Conclusions} \label{sec:conclusions}

  PSF reconstruction provides an excellent medium scale inverse problem to test state-of-the-art MCMC algorithms for posterior estimation.
  This work shows how modifying the hierarchical Gibbs sampler first presented in \cite{bardsley2012mcmc} can result in the MTC algorithm which is equivalent to the one derived in \cite{fox2015fast} and the hierarchical PC Gibbs algorithm.
  Both methods have their advantages: MTC having $\vect \theta$ decoupled from $\vect p$ makes it possible to sample $\vect p^k$ at the rate of the integrated autocorrelation time and makes analysis of the algorithm easier; where PC Gibbs is a straight-forward modification of the Hierarchical Gibbs algorithm with an easily tuned one-dimensional Metropolis-Hastings step, which can easily be tuned to be very efficient.
  We have provided statistical evidence that both algorithms are essentially equivalent for PSF reconstruction in terms of an estimator that measures computational effort per uncorrelated sample.

  This work contributes two novel aspects to the relevant literature.
  First is in the application of PSF reconstruction to X-ray imaging, which to our knowledge, has not appeared elsewhere in the inverse problems literature.
  This involved a novel approach to incorporating radial symmetry in prior modeling the PSF with a Gauss-Markov random field.
  The results illustrate the effectiveness of a sample-based approach on real data for uncertainty quantification.  
  The second contribution of this work is in new advances of addressing the autocorrelated $\delta$ component of a hierarchical Gibbs sampler in \cite{bardsley2012mcmc}. 
  The work builds upon \citep{agapiou2013aspects,agapiou2014analysis} by collapsing only the $\delta_k$ component of Gibbs and retaining the efficiency in sampling $\lambda_k$ gained by its dependence on $\vect p^k$.
  We showed that this fits into the framework of partial collapse presented in \cite{van2008partially}, and heeding their warnings of creating an improper sampler, we prove that our algorithm is still invariant with respect to the desired posterior density.
  Finally, MCMC methods were verified using both a synthetic test case and real data. 
  
  The PC Gibbs sampler is readily adapted to other linear Bayesian inverse problems modelled hierarchically and has potential applications to more general prior modeling (i.e., non-conjugate priors).
  In general, this work provides evidence that in models with a Gaussian noise likelihood, it is advantageous to employ an MCMC transition that retains the dependence between the parameter defining the likelihood and the data, rather than completely decoupling them.

\section*{Acknowledgments}
  The authors would like to thank Peter Golubstov for helpful comments and suggestions on the work and manuscript.  
This manuscript has been authored by National Security Technologies, LLC, under Contract No. DE-AC52-06NA25946 with the U.S.~Department of Energy, National Nuclear Security Administration, [NNSA Subprogram Office funding source]. The United States Government retains and the publisher, by accepting the article for publication, acknowledges that the United States Government retains a non-exclusive, paid-up, irrevocable, worldwide license to publish or reproduce the published form of this manuscript, or allow others to do so, for United States Government purposes. The U.S.~Department of Energy will provide public access to these results of federally sponsored research in accordance with the DOE Public Access Plan (\url{http://energy.gov/downloads/doe-public-access-plan}). The views expressed in the article do not necessarily represent the views of the U.S.~Department of Energy or the United States Government. DOE/NV/25946--3373

  \bibliographystyle{siamplain}
  \bibliography{partially_collapsed_gibbs}

\end{document}